\documentclass[12pt,leqno,oneside,a4paper]{amsart}
\usepackage{graphicx}
\usepackage{amscd}
\usepackage{amsmath}
\usepackage{amsfonts}
\usepackage{amssymb}
\usepackage{mathtools}
\usepackage{verbatim}
\usepackage{comment}
\usepackage{color}
\usepackage[usenames,dvipsnames,table,xcdraw]{xcolor}
\usepackage{pdfsync}
\usepackage{bbm}
\usepackage{dsfont}
\usepackage[colorlinks]{hyperref}
\usepackage{enumerate}
\usepackage{booktabs}
\usepackage{float}
\usepackage{wrapfig}

\usepackage{tikz}
\numberwithin{equation}{section}
\theoremstyle{plain}
\newtheorem{theorem}{Theorem}[section]
\newtheorem{lemma}[theorem]{Lemma}

\newtheorem{conjecture}[theorem]{Conjecture}

\theoremstyle{definition}
\newtheorem{definition}[theorem]{Definition}
\newtheorem{condition}[theorem]{Condition}

\newtheorem{example}[theorem]{Example}

\theoremstyle{remark}

\newtheorem{remark}[theorem]{Remark}

\newtheorem{case[theorem]}{Case}

\newcommand*{\e}[1]{\text{e}^{#1}}

\definecolor{blue}{rgb}{0,0,1}

\definecolor{red}{rgb}{1,0,.2}

\DeclareSymbolFont{extraup}{U}{zavm}{m}{n}
\DeclareMathSymbol{\varheart}{\mathalpha}{extraup}{86}
\DeclareMathSymbol{\vardiamond}{\mathalpha}{extraup}{87}

\newcommand{\alert}{}
\newcommand{\structure}{}
\newcommand{\yboxduma}{}
\newcommand{\ybox}{}
\newcommand{\pause}{}
\newcommand{\cb}{}
\newcommand{\clrgreen}{}

\date{{\today}}

\begin{document}

\pagestyle{myheadings}

\title[Dimension theory of some non - Markovian repellers]{Dimension Theory of some non-Markovian repellers Part I: A gentle introduction}

\author{Bal\'azs B\'ar\'any}
\address[Bal\'azs B\'ar\'any]{Budapest University of Technology and Economics, Department of Stochastics, MTA-BME Stochastics Research Group, P.O.Box 91, 1521 Budapest, Hungary}
\email{balubsheep@gmail.com}

\author{Micha\l\ Rams}
\address[Micha\l\ Rams]{Institute of Mathematics, Polish Academy of Sciences, ul. \'Sniadeckich 8, 00-656 Warszawa, Poland}
\email{rams@impan.pl}

\author{K\'aroly Simon}
\address[K\'aroly Simon]{Budapest University of Technology and Economics, Department of Stochastics, Institute of Mathematics, 1521 Budapest, P.O.Box 91, Hungary} \email{simonk@math.bme.hu}

\thanks{The research of B\'ar\'any and Simon was partially supported by the grant OTKA K123782. B\'ar\'any acknowledges support also from NKFI PD123970 and the J\'anos Bolyai Research Scholarship of the Hungarian Academy of Sciences. Micha\l\ Rams was supported by National Science Centre grant 2014/13/B/ST1/01033 (Poland). This work was partially supported by the  grant  346300 for IMPAN from the Simons Foundation and the matching 2015-2019 Polish MNiSW fund.}

\begin{abstract}
Michael Barnsley introduced a family of  fractals sets which are repellers of
piecewise affine systems. The study of these fractals was motivated by certain problems that arose in fractal image compression but the results we obtained can be applied for the computation of the Hausdorff dimension
of the graph of some functions, like generalized Takagi functions and fractal interpolation functions.

In this paper we introduce this class of fractals and present the tools in the one-dimensional dynamics and nonconformal fractal theory that are needed to investigate them.  This is the first part in a series of two papers. In the continuation  there will be more proofs and we apply the tools introduced here to study some fractal function graphs.

\end{abstract}

\date{\today}
\maketitle

\section{Introduction}

This is a paper in the intersection of fractal geometry and dynamical systems. Dynamical systems provide us with beautiful and interesting examples of sets, fractal geometry gives us the language to describe them, and both theories give us tools. Tools to understand the geometric properties of those sets, tools to understand the dynamical properties, and most interesting of all -- the relations between the two.

This is a paper about tools. Yeah, sure, we will prove some theorem eventually (in the second part of this paper) -- but it is just a pretext. Our real goal is to describe the process of understanding the geometric behaviour of a dynamical system, starting from understanding the simplest possible models (conformal uniformly hyperbolic iterated function systems with separation properties) and then throwing out the training wheels, until we get to piecewise affine maps with quite general symbolic description (not necessarily subshifts of finite type).

And, most of all, this is a survey. While the simple models are in the books (the classical positions by Falconer \cite{falconer2004fractal} and by Mattila \cite{mattila1999geometry}), the modern theory of affine iterated function systems is not in books yet, and neither is Hofbauer's theory. We aren't going to be able to describe all the details, for sure, but we try to at least provide the main ideas and most useful formulas, and also the literature for further reading.

Fine, let's present the hero of our story.

\section{Barnsley's skew product maps}\label{z49}

  \begin{wrapfigure}{R}{4cm}
\centering
  \includegraphics[width=5cm]{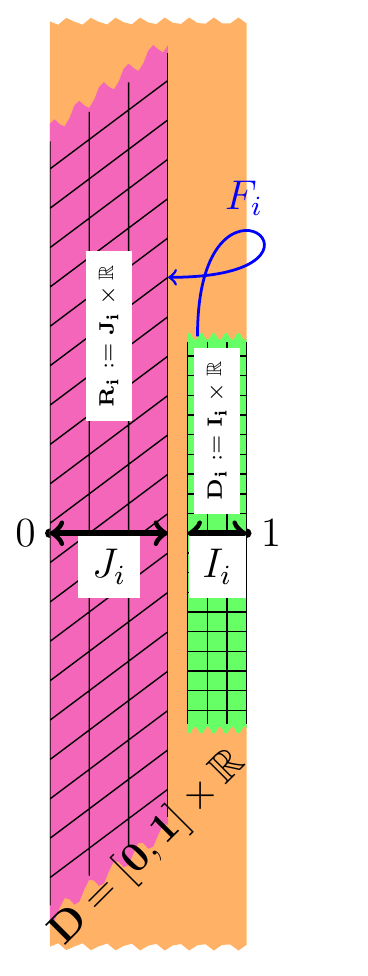}
  \label{z45}
\end{wrapfigure}

In order to define a piecewise affine and piecewise expanding skew product map $F$ on the plane which sends the vertical stripe
 $D:=[0,1]\times \mathbb{R}$ into itself, first
we partition the unit interval $\ybox{[0,1]=\bigsqcup\limits_{i=1}^{m} I_i}$.
Then we define $F:D\to D$ by
 \begin{flalign}\label{z46}
F(x,y):=F_i(x,y) \mbox{ if }(x,y)\in
  \alert{D_i:=I_i\times \mathbb{R}},&&
\end{flalign}
where for all $i=1, \dots ,m$
\begin{flalign}\label{z47}
   F_i(x,y):=(\alert{f_i(x)},\structure{g_i(x,y)}),\
  \mbox{ for } (x,y)\in D_i &&
 \end{flalign}
 and
 $f_i:I_i\to J_i \subset [0,1]$ (see Figure \ref{z51})
and $g_i:D_i\to\mathbb{R}$ and for $|\lambda_i|,|\gamma_i|>1$ let
 \begin{flalign}\label{z48}
 f_i(x):=\gamma_ix+v_i,\  g_i(x,y)=a_ix+\lambda_iy+t_i.
 &&
 \end{flalign}
Throughout this note we always assume:\newline
\textbf{Principal assumption} The map $f:[0,1]\to[0,1]$
 \begin{flalign}\label{y93}
f(x):=f_i(x),\mbox{ if } x\in I_i
\quad \mbox{ is transitive,}
&&
\end{flalign}
 that is $f$ has an orbit which is dense in $[0,1]$.
We call the repeller of $F:D\to D$ (which is the graph of a function) \yboxduma{Barnsley repeller} and we denote it by $\Lambda$.
We call $F$ Barnsley's skew product map.
Let $\mathfrak{S}=\bigcup_{i=1}^M\partial I_i$ the singularity set and let $\mathfrak{S}_\infty=\bigcup_{n=0}^{\infty}
f^{-n}(\mathfrak{S})$.
It was pointed out by Barnsley that $\Lambda$ is the graph of a function $G:[0,1]\setminus \mathfrak{S}_\infty:\to\mathbb{R}$
which is defined by
 \begin{flalign}\label{z34}
G(x)=z\text{, where }\{F^n(x,z)\}_{n=1}^{\infty}\text{ is bounded.}&&
\end{flalign}

  \begin{figure}
  \centering
  \includegraphics[width=6cm]{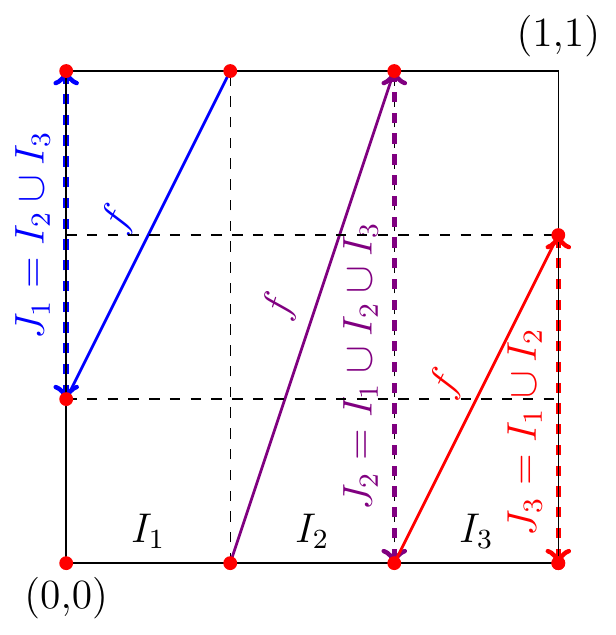}\qquad
\includegraphics[width=6cm]{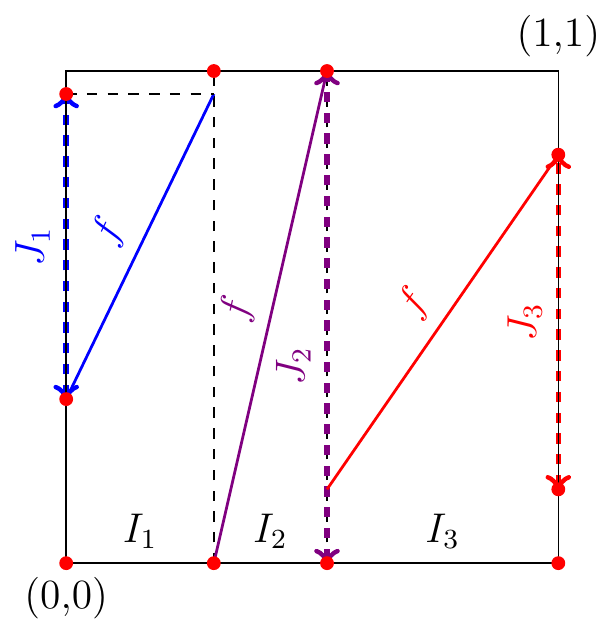}
\caption{$f$ is Markov on the left hand-side and non-Markov on the righ-hand side.}\label{z51}
\end{figure}


\section{The Hausdorff and box dimensions}
For a $d \geq 1$ let
 $A \subset \mathbb{R}^d$ be a set of zero Lebesgue measure and let $\nu$
 be a measure which is singular with respect to the Lebesgue measure $\mathcal{L}_d$.
 Then the size of $A$  and $\nu$ can be expressed by their  fractal dimensions.
\subsection{Fractal dimensions of sets}
The most common fractal dimensions are the Hausdorff and the  box dimensions:

 \begin{definition}[Hausdorff dimension]\label{z99}
   Let $A \subset \mathbb{R}^d$. then
 \begin{equation}\label{z98}
   \dim_{\rm H} A :=\inf\left\{\alpha:
   \forall \varepsilon>0, \exists \left\{U_i\right\}_{i=1}^{\infty }, \mbox{ such that }
  A \subset \bigcup\limits_{i=1}^{\infty }U_i,\
   \sum\limits_{i=1}^{\infty }|U_i|^\alpha<\varepsilon
   \right\},
 \end{equation}
 where $|U_i|$ is the diameter of $U$.
 \end{definition}
Equivalently in a more traditional way we can first define the $t$-dimensional Hausdorff measure
\begin{equation}\label{z97}
\mathcal{H}^t(A )=%
\sup\limits_{\delta \to 0}%
\inf\left\{
\boxed{\sum\limits_{i=1}^{\infty }|E_i|^t}:%
\Lambda \subset \bigcup_{i=1}^{\infty } E_i,%
|E_i|<\delta
\right\},
\end{equation}
then we write see (Figure \ref{z95})
\begin{equation}\label{z98}
    \dim_{\rm H} A :=\inf\left\{t:\mathcal{H}^t(A )=0 \right\}
    =\sup\left\{t:\mathcal{H}^t(A )=\infty \right\}.
\end{equation}
\begin{figure}[h!]
\includegraphics[height=5cm]{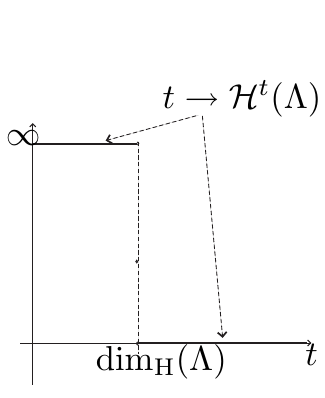}
\label{z95}
\end{figure}

Another very popular notion of fractal dimension is the box dimension:
\begin{definition}\label{z96}

 $ \dim_{\rm B}A$
\end{definition}

Let $E\subset \mathbb{R}^d$, $E\ne\emptyset $, bounded.
$N_\delta (E)$ be the smallest number of sets of diameter $\delta $ which can cover $E$. Then the lower and upper box dimensions of $E$:
\begin{equation}\label{15}
  \underline{\dim}_{\rm B}(E):=\liminf_{r\to 0}\frac{\log N_\delta (E)}{-\log \delta },
\end{equation}

\begin{equation}\label{16}
  \overline{\dim}_{\rm B}(E):=\limsup_{r\to 0}\frac{\log N_\delta (E)}{-\log \delta }.
\end{equation}

If the limit exists then we call it the box dimension of $E$ and we denote it by $\dim_{\rm B}(E) $.

\subsection{Hausdorff dimension of measures}
The Hausdorff dimension of a measure $\mu$ is the best lower bound on the Hausdorff dimension of a sets having large $\mu$ measures. Depending on what
"large" means we define

\begin{definition}\label{z89}
Let $\mu$ be a Borel measure on $\mathbb{R}^d$ such that
$0<\mu(\mathbb{R}^d)<\infty $.
  \begin{description}
  \item[(a)] Lower Hausdorff dimension of $\mu$ is:
  \alert{$\mathrm{dim}_*(\mu):=\inf\left\{\dim_{\rm H} A: \mu(A)>0\right\}$},
  \item[(b)] \structure{Upper  Hausdorff dimension of $\mu$:}
\alert{$\mathrm{dim}^*(\mu):=\inf\left\{\dim_{\rm H} A: \mu(A^c)=0\right\}$}.
\item[(c)] The lower and the upper local dimension of the measure $\mu$ are:
 \begin{equation}\label{O63}
 \alert{ \underline{\dim}(\mu,x):=  \liminf\limits_{r\to 0} \frac{\log \mu(B(x,r))}{\log r}}
  \end{equation}
  and
  \begin{equation}\label{z85}
 \structure{\overline{\dim}(\mu,x):=     \limsup\limits_{r\to 0} \frac{\log \mu(B(x,r))}{\log r}}
  \end{equation}
  We say that the measure $\mu$ is exact dimensional if for $\mu$-almost all $x$
   $\lim\limits_{r\downarrow 0} \frac{\log \mu(B(x,r))}{\log r}$ exists and equals to a constant.
\end{description}
\end{definition}
\begin{lemma}\label{z86}
Let $\mu$ be a measure like in \eqref{z89}. Then
  \begin{equation}\label{z90}
     \alert{\dim_*\mu=\mathrm{ess inf}_{x\sim \mu}\underline{\dim}(\mu,x)}, \quad
     \structure{\dim^*\mu=\mathrm{ess sup}_{x\sim \mu}\underline{\dim}(\mu,x)}
  \end{equation}
\end{lemma}

\section{Self-similar Sets}

 From now on we work on $\mathbb{R}^d$. Let  $m \geq 2$ and
  $O_1, \dots ,O_m\in O(d)$ orthogonal matrices and $r_1, \dots ,r_m\in(0,1)$ and $t_1, \dots ,t_m\in\mathbb{R}^d$.
  Then
  \begin{equation}\label{O74}
   \ybox{ \mathcal{S}:=\left\{S_i(x)=r_i \cdot O_i x+t_i\right\}_{i=1}^{m}}
  \end{equation}
is called a  self-similar Iterated Function System on $\mathbb{R}^d$.

\begin{figure}[H]
  \centering
  \includegraphics[width=15cm]{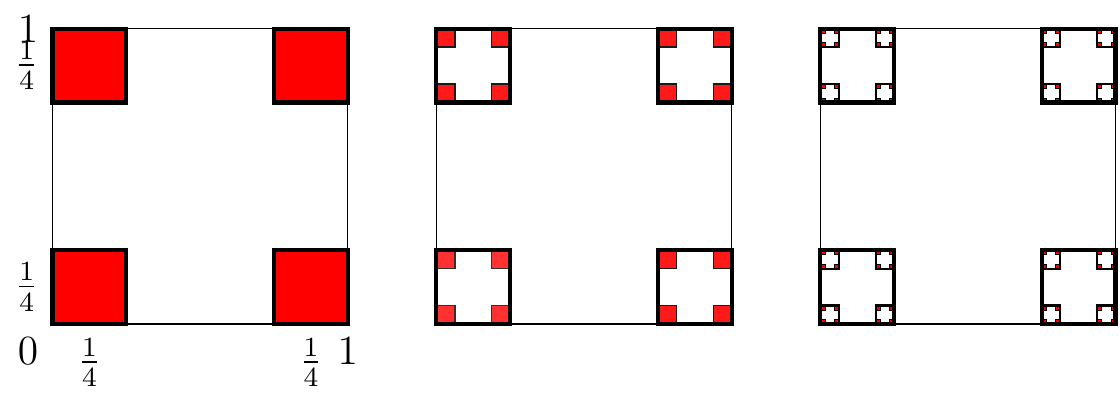}
\caption{The Four-Corner Cantor set $C\left(\frac{1}{4}\right)$}
\label{z94}
\end{figure}

Let $B:=B(x,R)$ be a closed ball, where $R$ is large. Then
\begin{equation}\label{O25}
  \forall i=1, \dots ,m:\quad
\alert{S_i(B) \subset B}.
\end{equation}
Hence the the following is a nested sequence of compact sets:
$$
\left\{\alert{\bigcup_{i_1 \dots i_n}S_{i_1 \dots i_n}B}\right\}_{n=1}^{\infty },
$$
where we use throughout the paper the notation: $S_{i_1 \dots i_n}:=S_{i_1}\circ \cdots \circ S_{i_n} $.
The \yboxduma{attractor} of  our IFS $\mathcal{S}$ is
\begin{equation}\label{O26}
\ybox{\alert{\Lambda}:=  \bigcap_{n=1}^{\infty }\bigcup_{i_1 \dots i_n}S_{i_1 \dots i_n}B},
\end{equation}
which is independent of $B$ as long as $B$ satisfies \eqref{O25}.
\begin{example}[Four Corner Set]\label{z93}
  Figure \ref{z94} shows the first three iterations of a famous self-similar set, called the Four Corner Cantor set.  Here  $B=[0,1]^2$ and
  $$
  S_i(x,y)=\frac{1}{4}(x,y)+\mathbf{t}_i, \mbox{ for }
  \mathbf{t}_1=(0,0),\
  \mathbf{t}_2=\left(\frac{3}{4},0\right),
   \mathbf{t}_3=\left(\frac{3}{4},\frac{3}{4}\right),
    \mathbf{t}_3=\left(0,\frac{3}{4}\right).
  $$

\end{example}
In the general case,
we code the points of the attractor by the elements of the
 symbolic space:
\begin{equation}\label{O23}
  \Sigma:=\left\{1, \dots ,m\right\}^\mathbb{N}.
\end{equation}

The \yboxduma{natural projection} is
$\Pi:\Sigma\to\Lambda$:
\begin{equation}\label{O22}
\ybox{\Pi(\mathbf{i}):= \lim\limits_{n\to\infty} S_{i_1 \dots i_n}(0).}
\end{equation}
On Figures \ref{z92} and \ref{z91} we indicate how this coding  works.
\begin{figure}[H]
  \includegraphics[width=12cm]{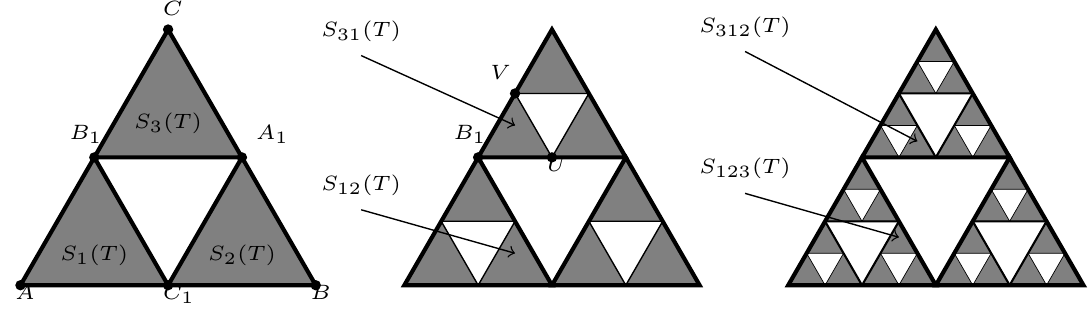}\\
  \caption{the Sierpi\'nski gasket: $S_{312}(x):=S_3\circ S_1\circ S_2(x)=S_3(S_1(S_2(x)))$}\label{z92}
\end{figure}
$S_i$ are translations of the appropriate \yboxduma{homothety-transformatons} of the form: $$\ybox{S_i(x)=\frac{1}{2}x+t_i}.$$
The sets $\left\{S_i(T)\right\}_{i=1}^{3}$ in the previous examples ar the first cylinders, the sets $\left\{S_{i,j}(T)\right\}_{i,j=1}^{3}$ are the second cylinders an so on.
\begin{figure}[H]
  \includegraphics[width=12cm]{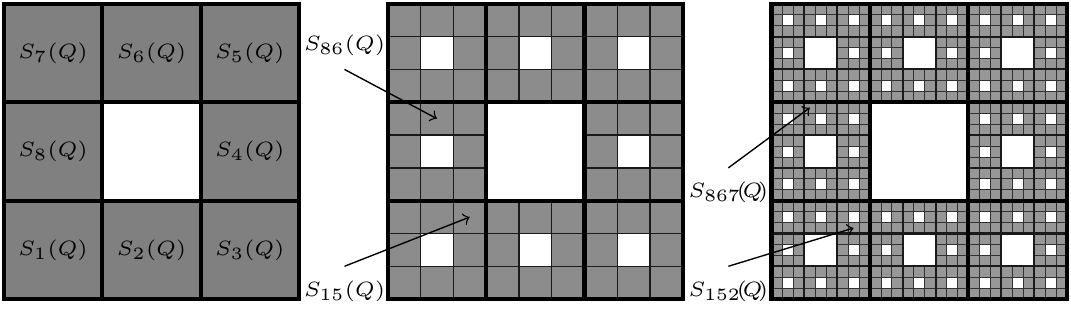}\\
  \caption{The third approximation of the Sierpi\'nski carpet}\label{z91}
\end{figure}
In both of the previous examples the cylinders were not disjoint but their interior were disjoint. This results that the cylinders are well separated.
\begin{definition}[ SSP,OSC,SOSC]\label{z42}
Here we define three important separation conditions. These will be used in much more general setup then the self-similar IFS.
  \begin{description}
    \item[(a)] If $S_i(\Lambda)\cap S_j(\Lambda)= \emptyset $ for all $i\ne j$ the we say that the Strong Separation Property (SSP) holds. (Like in the case of the Four Corner Cantor set.)
    \item[(b)] If there exists a bounded  open set $V$ such that
    \begin{enumerate}
      \item $S_i(V) \subset V$ for all $i=1, \dots ,m$
      \item $S_i(V)\cap S_j(V)= \emptyset $ for all $i\ne j$
    then we say that the Open Set Condition (OSC) holds like in the case of
    the Sierpi\'nski gasket and Sierpi\'nski carpet. Here $V$ is the interior of the right triangle and the unit square respectively.
    \item[(c)] If the OSC holds with an open set $V$ satisfying $V\cap\Lambda\ne \emptyset $, where $\Lambda$
 is the attractor, then we say that
 the Strong Open Set Condition (SOSC) holds.
     \end{enumerate}
  \end{description}
  The OSC and SOSC are equivalent for self-similar (and also for self-conformal) IFS.
\end{definition}

  Now we present a heuristic argument
in order to guess the Hausdorff dimension of the attractor  $\Lambda$
in the case when
  the cylinders are disjoint (that is when SSP holds):

  We will use the following fact: it is immediate from the definition that  for any $r>0$ we have:
  \begin{equation}\label{O24}
    \structure{\mathcal{H}^s(r \cdot A):=r^s \cdot \mathcal{H}^s(A)}.
  \end{equation}

  Since this is only a heuristic argument we assume that for the appropriate $s$,
  (that is for the $s$ satisfying $s=\dim_{\rm H} \Lambda$)
   the $s$-dimensional Hausdorff measure of the attractor $\Lambda$ has positive and finite. Then
  \begin{eqnarray*}
    \structure{\mathcal{H}^s(\Lambda)} &=&\sum\limits_{i=1}^{m}  \mathcal{H}^s(S_i\Lambda) \\
     &=& \sum\limits_{i=1}^{m} r_{i}^{s} \structure{\mathcal{H}^s(\Lambda)}.
  \end{eqnarray*}
  By the assumption above, we can divide by $\mathcal{H}^s(\Lambda)$. This yields that:
  \begin{equation}\label{O72}
    \ybox{\sum\limits_{i=1}^{m}r_{i}^{\alert{s}}=1.}
  \end{equation}

Even if $\mathcal{S}$ does not satisfy any of the previous assumptions we can define $s$ as the solution of \eqref{O72}.
  \begin{definition}
  Let $\mathcal{S}$ be a self-similar IFS of the form \eqref{O74}.
  The similarity dimension  $\dim_{\rm S}(\Lambda):=s $  where $s$ is the
 unique solution of \eqref{O72}. That is $\ybox{\sum\limits_{i=1}^{m}r_{i}^{\alert{s}}=1}.$ Sometimes we also say that $s$ is the similarity dimension of the attractor.
  \end{definition}
  Clearly,
  \begin{equation}\label{O71}
   \ybox{ \dim_{\rm H} (\Lambda) \leq \dim_{\rm S}(\Lambda)}.
  \end{equation}
  However "$=$" does not always hold:

Let $\Lambda_{1/3}$ be the attractor the $\mathcal{S}^{1/3}$ from \eqref{O33}: $$\ybox{\mathcal{S}^{1/3}=\mathcal{S}:=\left\{\frac{1}{3}x,\frac{1}{3}x+1,\frac{1}{3}x+3\right\}}.$$
 Then
  \begin{equation}\label{O70}
    \dim_{\rm B} (\Lambda_{1/3}) <0.9<1= \dim_{\rm S}(\Lambda_{1/3}).
  \end{equation}
This is so because in this case
$$
S^{1/3}_0\circ S^{1/3}_3\equiv S^{1/3}_1\circ S^{1/3}_0
$$
so there is an \yboxduma{exact overlap}.

\begin{theorem}[Hutchinson's-Moran Theorem \cite{moran_old} and \cite{Hutchinson}]
Let $\mathcal{S}:=\left\{S_1, \dots ,S_m\right\}$ be a self-similar IFS on $\mathbb{R}^d$ with contraction ratios $r_1, \dots ,r_m$ and similarity dimension $s$. We assume that the OSC (Open Set Condition) holds.

then
\begin{description}
  \item[(a)] $\ybox{\dim_{\rm H} \Lambda=s}$, even we have
  \item[(b)] $0<\mathcal{H}^s(\Lambda)<\infty $,
  \item[(c)] $\mathcal{H}^s\left(S_i(\Lambda)\cap S_j(\Lambda)\right)=0$ for all $i\ne j$.
\end{description}
\end{theorem}

\begin{theorem}[Falconer]\label{O69}
    The Hausdorff- and box-dimensions are the same for any self-similar set.
  \end{theorem}

 The following problem is one of the most interesting open problems in Fractal Geometry:

\begin{conjecture}[Complete Overlap Conjecture]\label{y98}
  Let $s$ be the similarity dimension and let $\Lambda$ be the attractor
of a self-similar IFS $\mathcal{S}=\left\{S_i\right\}_{i=1}^{m}$ on $\mathbb{R}$. Then
\begin{equation}\label{O27}
  \ybox{\dim_{\rm H} (\Lambda)<\min\left\{d, s\right\}}
  \Longleftrightarrow
  \exists \mathbf{i},\ \mathbf{j}\in\Sigma^*, \mathbf{i}\ne\mathbf{j} \mbox{ s.t. }\ybox{S_{\mathbf{i}}\equiv S_{\mathbf{j}}}.
\end{equation}
\end{conjecture}
In $\mathbb{R}^2$ the conjecture does not hold. The following example was introduced by M. Keane, M. Smorodinsky and B. Solomyak \cite{keane1995morphology} and played a very important role in the study of self-similar fractals with overlapping construction.
\begin{example}\label{z84}
  For every $\lambda \in (\frac{1}{4},\frac{2}{5})$ consider the following self-similar set:\pause \

\begin{equation*}
{\cb  \ybox{\widetilde{\Lambda} _\lambda} :=\left\{\alert{\sum\limits_{i=0}^{\infty }a_i\lambda^{i}}:%
a_i\in \left\{0,1,3\right\}
  \right\}.}
\end{equation*}

Then ${\cb\widetilde{\Lambda}_\lambda }$ is the attractor of the one-parameter ($\lambda $) family IFS:\pause \

\begin{equation}\label{O33}
  {\cb \mathcal{S}^\lambda:=\left\{\ybox{S_{i}^{\lambda }(x)}:=\lambda \cdot
x+i\right\}_{i=0,1,3}}
\end{equation}

To normalize it we write $\Lambda_\lambda:=\frac{1-\lambda}{3} \cdot \widetilde{\Lambda}_\lambda$. It was proved by Solomyak \cite{solomyak1998measure}
that for Lebesgue almost all $\lambda>\frac{1}{3}$ (that is when the similarity dimension is greater than one)  we have
\begin{equation}\label{z82}
  \dim_{\rm H} \Lambda_\lambda=1.
\end{equation}
Fix a $\lambda$ slightly greater than $1/3$ for which \eqref{z82} holds and consider the product set $C_\lambda:=\Lambda_\lambda\times [0,1]$ (see Figure \ref{z81}). Then for $\lambda\in\left(\frac{1}{3},\frac{1}{\sqrt{6}}\right)$we have
$$\dim_{\rm H} C_\lambda=1+\frac{\log 2}{-\log \lambda}
<
\min\left\{2,\frac{\log 6}{-\log \lambda}\right\}=\min\left\{2,\dim_{\rm Sim}(\mathcal{S}) \right\}.
$$
Since there are uncountably many $\lambda$ like this, and complete overlap can happen only for countably many $\lambda$, we get that dimension drop occur in higher dimension not only when we have complete overlaps.
\begin{figure}[H]
 \includegraphics[width=7cm]{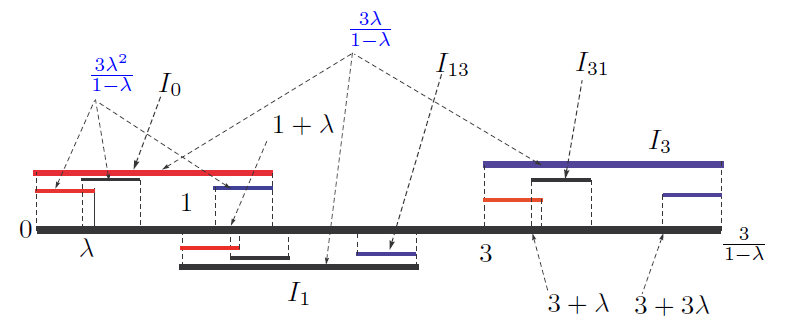}
  \includegraphics[height=6cm]{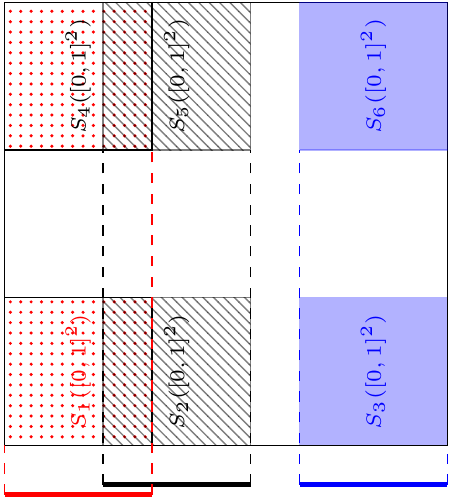}\\
  \caption{$\widetilde{\Lambda}_\lambda$ and $C_\lambda:=\Lambda_\lambda\times [0,1]$}\label{z81}
\end{figure}

\end{example}

\subsection{Self-similar measures}\label{z41}
Analogously to the self-similar sets, we can define the self-similar measures:
\begin{definition}\label{z09}
 Given an $m \geq 2$, $\ybox{\mathcal{S}=\left\{S_1, \dots ,S_m\right\}}$
self-similar IFS on $\mathbb{R}^d$ with contraction ratios: $\alert{r_1, \dots ,r_m}$ and we are given a probability
vector $\alert{\mathbf{p}=\left(p_1, \dots ,p_m\right)}$. Now we define the
self-similar measure $\nu=\nu_{\mathcal{S}, \mathbf{p}}$
which corresponds to $\mathcal{S}$ and $\mathbf{p}$:
\begin{equation}\label{z07}
  \ybox{\nu_{\mathcal{S}, \mathbf{p}}:=\Pi_*\left(\mathbf{p}^{\mathbb{N}}\right):=\mu\circ\Pi^{-1}}.
\end{equation}
\end{definition}
Then $\nu_{\mathcal{S}, \mathbf{p}}$ is the unique probability Borel  measure
satisfying
\begin{equation}\label{z08}
  \nu_{\mathcal{S}, \mathbf{p}}(H)=\sum\limits_{k=1}^{m}
p_i \cdot \nu_{\mathcal{S}, \mathbf{p}}\left(S_{i}^{-1}(H)\right),
\end{equation}
for every Borel set $H$.

    Let $\ybox{\nu:=\nu_{\mathcal{S},\mathbf{p}}}$ be the invariant measure for the self-similar IFS on $\mathbb{R}^d$:
 \begin{equation}\label{O65}
    \alert{\mathcal{S}:=\left\{S_i(x)=r_i \cdot O_i x+t_i\right\}_{i=1}^{m}}.
  \end{equation}
  Below we give a heuristic argument to show that \alert{if the OSC holds} then
   the Hausdorff dimension of $\nu$ is equal to  the similarity dimension of $\nu$, which is defined by:
  \vspace{-0.3cm}
 \begin{equation}\label{O64}
 \ybox{  \dim_{\rm {Sim}}\nu:=\frac{\sum\limits_{i=1}^{m}p_i \log p_i}{\sum\limits_{i=1}^{m} p_i\log r_i}=\frac{\mbox{entropy}}{\mbox{Lyapunov exponent}}.}
 \end{equation}

\begin{lemma}\label{z88}
  $\mathcal{S}$ and $\mathbf{p}$ as above and we assume that the OSC holds. Then
  \begin{equation}\label{z87}
    \dim_{\rm H} \nu=\dim_{\rm Sim} \nu.
  \end{equation}
\end{lemma}
\begin{proof}[Heuristic Proof]
Let $I$ be a large interval such that $S_i(I) \subset I$ for all $i=1, \dots ,m$ and we write $ I_{i_1 \dots i_n}:=S_{i_1 \dots i_n}I$ for  the level $n$ cylinder intervals.
  It follows from Birkhoff's Ergodic Theorem that in this case the limit in
  \eqref{O63} and \eqref{z85} exist. That is, Lemma \ref{z86} indicates that
     for a $\nu$-typical $x=\Pi(\mathbf{i})$, $\mathbf{i}\in\Sigma$:
  \begin{eqnarray*}
 \ybox{\dim_{\rm H} \nu }  &=&\lim\limits_{n\to\infty} \frac{\log \nu( I_{i_1 \dots i_n})}{\log|I_{i_1 \dots i_n}|}\stackrel{\mathrm{def}}{=}\lim\limits_{n\to\infty} \frac{\log p_{i_1 \dots i_n}}{\log r_{i_1 \dots i_n}} \\
    &=&\frac{\lim\limits_{n\to\infty} \frac{1}{n}\sum\limits_{k=1}^{n}\log p_{i_k}}
    {\lim\limits_{n\to\infty} \frac{1}{n}\sum\limits_{k=1}^{n}\log r_{i_k}}
    \stackrel{\mathrm{LLN}}{=}\ybox{\frac{\sum\limits_{i=1}^{m}p_i \log p_i}{\sum\limits_{i=1}^{m} p_i\log r_i}}=\dim_{\rm {Sim}}\nu,
  \end{eqnarray*}
   where LLN means Law of Large Numbers. Here we used the notations: $p_{i_1 \dots i_n}:=p_{i_1} \cdots p_{i_n}$ and
   $r_{i_1 \dots i_n}:=r_{i_1} \cdots r_{i_n}$
\end{proof}

\subsubsection{Hochman Theorem}
Let $\mathcal{S}=\left\{S_i\right\}_{i=1}^{m}$ be a self-similar IFS on $\mathbb{R}$ with contraction ratios $\left\{r_i\right\}_{i=1}^{m}$.
Let $\Delta_n(\mathcal{S})$ be
the smallest distance between the left end points of
two level $n$ cylinders having the same length.
More formally,  $\Delta_n(\mathcal{S})$ is
 the minimum of $\Delta(\pmb{\omega},\pmb{\tau})$ for distinct
 $\pmb{\omega},\pmb{\tau}\in\Sigma_n$, where
\[
\Delta(\pmb{\omega},\pmb{\tau})=\left\{\begin{array}{cc}
             \infty & S_{\pmb{\omega}}'(0)\neq S_{\pmb{\tau}}'(0) \\
             \left|S_{\pmb{\omega}}(0)-S_{\pmb{\tau}}(0)\right| & S_{\pmb{\omega}}'(0)=S_{\pmb{\tau}}'(0).
           \end{array}\right.
\]

\begin{condition}[HESC]\label{y96}
We say that the self-similar IFS $\mathcal{S}$ satisfies \underline{Hochman's} \underline{exponential} \underline{separation condition} (HESC) if there exists an $\varepsilon>0$ and an $n_k\uparrow \infty $ such that
\begin{equation}\label{k78}
  \Delta_{n_k}>\varepsilon^{n_k}.
\end{equation}
\end{condition}
Hochman proved the following very important assertion in \cite[Theorem~1.1]{hochman2012self}.
\begin{theorem}[Hochman]\label{k79}
	Assume that  $\mathcal{S}=\left\{S_i\right\}_{i=1}^{m}$ is a self-similar IFS on $\mathbb{R}$ which satisfies Hochman's exponential separation condition. Let $\mathbf{p}=(p_1, \dots ,p_N)$ be an arbitrary probability vector.
Then
\begin{equation}\label{k80}
  \dim_H\left(\nu_{\mathcal{S},\mathbf{p}}\right)=\min\left\{1,
\dim_{\rm Sim}\nu \right\},
\end{equation}
\end{theorem}

\begin{remark}[Relation to the Compete Overlaps Conjecture]\label{y99}
Although Hochman's Theorem does not solve the Compete Overlaps Conjecture (Conjecture \ref{y98}) but it makes a very significant progress towards it.

\begin{itemize}
  \item Exact overlap means that $\Delta _n=0$ for some $n$.
\item
 If the OSC holds then $\alert{\Delta _n\to 0}$ exactly  exponentially fast.
  \item \alert{$\Delta _n\to 0$ at least exponentially fast} always holds. Namely: $\alert{\#\left\{|\mathbf{i}|=n\right\}=m^n}$. On the other hand: \structure{$\#\left\{r_{\mathbf{i}}:|\mathbf{i}|=n\right\}$ is polynomially large} ($r_{\mathbf{i}}$ was the contraction ration of $S_{\mathbf{i}}$).
So, there exist distinct \clrgreen{$\mathbf{i},\mathbf{j}$ of length $n$} with \clrgreen{$r_{\mathbf{i}}=r_{\mathbf{j}}$} and with
      \yboxduma{with exponentially small } $\alert{\left|S _{\mathbf{i}}(0)-S _{\mathbf{j}}(0)\right|}$.
\item
However, in case of a dimension drop, that is if we can find a probability vector $\mathbf{p}$ such that $\dim_{\rm H} \nu_{\mathcal{S},\mathbf{p}}<\min\left\{1,\dim_{\rm S} \nu\right\}$ then
$\Delta_n\to 0 $ super exponentially fast. That is
$$
\lim\limits_{n\to\infty} -\frac{1}{n}\log\Delta_n
=\infty .$$
\end{itemize}
\end{remark}
The following theorem shows that Hochman's theorem solves the  Complete Overlap Conjecture in some cases:
\begin{theorem}[Hochman]
 For an self-similar IFS on the line with algebraic parameters we have
 either exact overlaps, or no dimension drop:  $\dim_{\rm H} \Lambda =\min\left\{1,\dim_{\rm S}\Lambda \right\}$.
 \end{theorem}

\section{Dimension of the self-conformal sets and measures when OSC holds}
We can extend a large part of the dimension theory of self-similar sets
to the so called self-conformal ones by using the notion of the  topological pressure.

\begin{definition}[Conformal IFS on the line]\label{z79}
 Let $\eta>0$ and $m>1$. We are given
$f_1, \dots ,f_m:[0,1]\to[0,1]$ satisfying the following conditions:

\begin{description}
\item[(a)] $f_i\in \mathcal{C}^{1+\eta}[0,1]$ for all $i=1, \dots ,m$,
  \item[(b)] $\exists\  0<c_1,c_2<1$ such that $c_1<|f'_i(x)|<c_2$ holds for all $i=1, \dots ,m$ and all $x\in [0,1]$.
\end{description}
Then we say that
\begin{equation}\label{z80}
  \mathcal{F}:=\left\{f_1, \dots ,f_m\right\}
\end{equation}
is a self-conformal  IFS. We can define the attractor, the symbolic space and the natural projection analogously as we did in \eqref{O26}, \eqref{O23} and \eqref{O22} respectively.
\end{definition}

A very important property of the self-conformal IFS the following:

\begin{theorem}[Bounded Distortion Property]
  Let $\mathcal{F}$ be as in Definition \ref{z79}. Then there exist
  $0<c_3<c_4$ such that for all $n$ and for all $(i_1, \dots ,i_n)\in(1, \dots ,m)^n$ and for all $x,y\in[0,1]$ we have
  \begin{equation}\label{z78}
    c_3<\frac{f'_{i_1, \dots ,i_n}(x)}{f'_{i_1, \dots ,i_n}(y)}<c_4,
  \end{equation}
\end{theorem}
The proof is available in \cite{Pesin97}.
Our aim is to calculate the Hausdorff dimension of the attractor.

\subsection{Hausdorff dimension of self-conformal sets when OSC is assumed}

\begin{theorem}\label{z69}
  Let $\mathcal{F}$ be a conformal IFS on $\mathbb{R}$ as in definition \ref{z79} and we assume that the OSC holds.
  Let $s_0$ be the root of the pressure formula that is we assume that \eqref{z71} holds. Then
  \begin{equation}\label{z68}
    \dim_{\rm H} \Lambda=s_0.
  \end{equation}
\end{theorem}

\begin{proof}
  First we prove that $\dim_{\rm H} \Lambda \leq s_0$. This is so, since
  the system of level $n$ cylinder intervals
  $\mathcal{I}_n:=\left\{f_{i_1 \dots i_n}([0,1])\right\}_{(i_1 \dots i_n)\in(1, \dots ,m)^n}$ gives a cover of as small diameter as we want if $n $ is large enough. Moreover, by Lagrange Theorem for suitable $x_{\pmb{\omega}}\in[0,1]$
  $$
  \sum\limits_{I\in\mathcal{I}_n}|I|^{s_0}=
  \sum\limits_{|\pmb{\omega}|=n}|f'_{\pmb{\omega}}
  (x_{\pmb{\omega}})|^{s_0} \leq \frac{1}{c_1c_3}\sum\limits_{|\pmb{\omega}|=n}\mu(\pmb{\omega})
  =\frac{1}{c_1c_3}.
  $$
That is $\mathcal{H}^{s_0}(\Lambda)<\infty $ consequently $\dim_{\rm H} \Lambda \leq s_0$.

 Now we prove that $\dim_{\rm H} \Lambda \geq s_0$.
Let $\mu$ be the Gibbs measure for the potential $\phi_{s_0}$
(defined in \eqref{z67}). Fix an arbitrary $\mathbf{i}\in \Sigma$. Then putting together \eqref{10}, \eqref{z71} and \eqref{z70} we obtain the following limit exists
$$\lim\limits_{n\to\infty} \frac{\log\Pi_*\mu(I_{i_1 \dots i_n})}
{\log|I_{i_1 \dots i_n}|}\equiv s_0.$$
That is the local dimension of the measure $\Pi_*\mu$ is equal to $s_0$
at all points of the attractor $\Lambda$. Hence $\dim_{\rm H} \Pi_*\mu=s_0$. This implies that $\dim_{\rm H} \Lambda \geq s_0$.
\end{proof}
We say that the measure $\mu$ in the previous proof is the natural measure for the IFS $\mathcal{F}$.

\subsection{Hausdorff dimension of an invariant measure and Lyapunov exponents}
\label{z54}Now we present the Lyapunov exponents for the classes of maps that occur in this paper.

\emph{ Ergodic measures for a piecewise monotone map on the interval.}
 Let  $\eta$ be an ergodic measure for a $T:[0,1]\to[0,1]$ piecewise monotonic map. Then the Lyapunov exponent $\chi(\eta)=\int \log|T'|d\eta$. It follows from Hoffbauer and Raith \cite[Theorem 1]{hofbauer_raith} that
        \begin{equation}\label{z44}
          \dim_{\rm H}\eta=\frac{h(\mu)}{\chi(\eta)}\quad \mbox{ if }
          \chi(\eta)>0.
        \end{equation}

\section{The Hausdorff dimension of self-affine sets}
\begin{definition}[Self-affine IFS and self-affine measures]\label{z39}
    We say that
    \begin{equation}\label{z38}
      \mathcal{F}:=\left\{f_1(x)=A_1x+t_1, \dots ,f_m(x)=A_mx+t_m\right\}
    \end{equation}
 is a self-affine IFS on $\mathbb{R}^d$ for a $d \geq 2$ if $A_1, \dots ,A_m$ are
    contractive non-singular $d\times d$ matrices and $t_1, \dots ,t_m\in\mathbb{R}^d$.
    The natural projection $\Pi$ from the symbolic  $\Sigma:=\left\{1, \dots ,m\right\}^{\mathbb{N}}$ space to the attractor $\Lambda$ (which is defined as in \eqref{O26})is defined as in the self-similar case: $\Pi(\mathbf{i}):=\lim\limits_{n\to\infty}
         f_{i_1}\circ\cdots\circ f_{i_n}(0)
         $.
The attractors of self-affine IFS  are called self-affine sets. The computation of the dimension of the self-affine
sets is much more difficult. Namely, in the self-similar case if the cylinders are well-separated  that is OSC holds (see Definition \ref{z42}) then
\begin{description}
  \item[(a)] The Hausdorff dimension of the attractor is equal to the similarity dimension $s$, which can be calculated merely  from the contraction ratios ( \eqref{O72} ), regardless the translations, as long as the cylinders remain well separated.
  \item[(b)] The appropriate  dimensional Hausdorff measure of the attractor is positive and finite.
 \item[(c)]  The Hausdorff and the box dimensions of self-similar sets are the same.
\end{description}
\begin{figure}
  \centering
  \includegraphics[width=\textwidth]{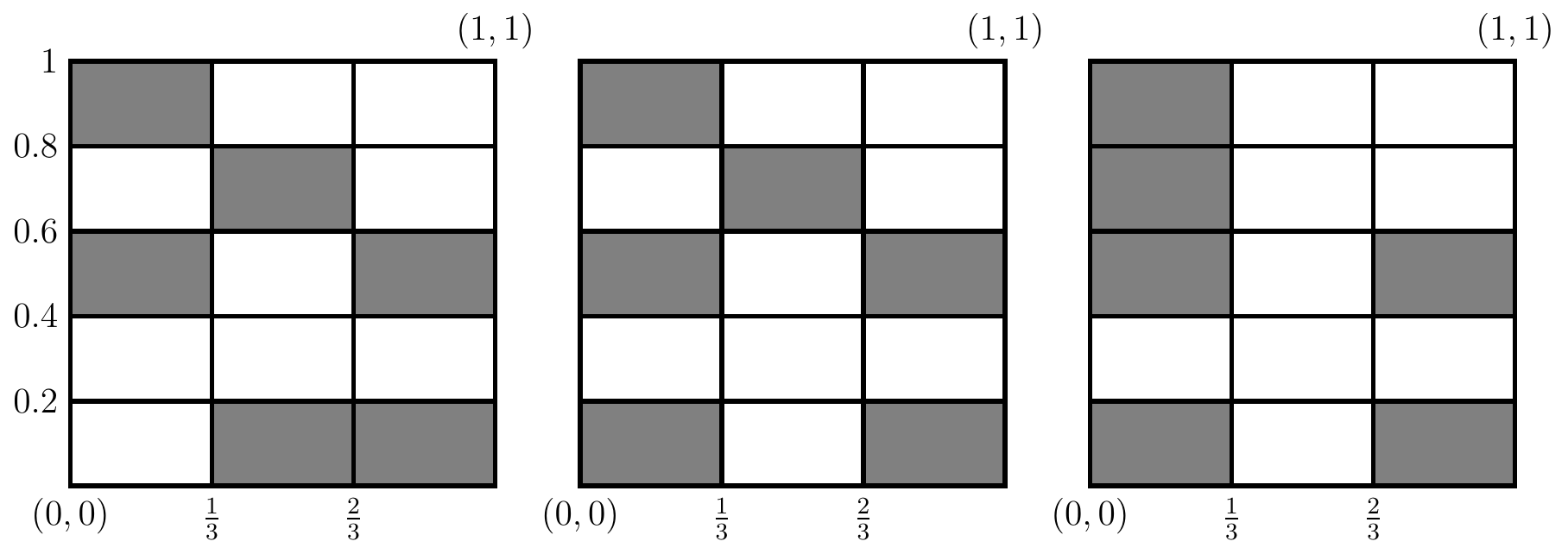}
  \caption{Left: $\alert{\dim_{\rm H} \Lambda^l=\dim_{\rm B}\Lambda^l=\dim_{\rm Aff}\Lambda } $
middle: $\structure{\dim_{\rm H} \Lambda^m<\dim_{\rm B}\Lambda^m=\dim_{\rm Aff}\Lambda^m  }$
right: $\dim_{\rm H} \Lambda^r<\dim_{\rm B}\Lambda^r<\dim_{\rm Aff}\Lambda^r  $
}\label{y81}
\end{figure}

In the self-affine case we will define the affinity dimension which replaces the similarity dimension. However, not any
of the assertions (a)-(c) hold for all self-affine sets with disjoint cylinders.
\begin{example}\label{y80}
On the left-hand side  Figure \ref{y81} we see three copies of the unit square.
Focus on the one which is on the left-hand side.
It contains six shaded rectangles of size $\frac{1}{3}\times \frac{1}{5}$. Denote their left bottom corners by $t_1, \dots ,t_6$ in any particular order. Then we define the IFS
$$\mathcal{F}^{l}:=\left\{f_i(x)=\left(
                                   \begin{array}{cc}
                                     \frac{1}{3} & 0 \\
                                     0 & \frac{1}{5} \\
                                   \end{array}
                                 \right) \cdot x+t_i
\right\}_{i=1}^{6}.$$
Let $\Lambda^l$ be the attractor of $\mathcal{F}^{l}$.
Clearly the first cylinders of $\mathcal{F}^{l}$ are the shaded rectangles on the Figure.
 We say that $\mathcal{F}^l$ and $\Lambda^l$ are generated by the left hand-side of the Figure \ref{y81}. We define $\mathcal{F}^m$, $\Lambda^m$ and
$\mathcal{F}^r$, $\Lambda^r$ respectively, generated by the rectangles in the middle and right-hand side unit squares on Figure \ref{y81}. These self affine sets belongs to the family of Bedford-McMullen carpets (see \cite{falconer2004fractal} for more details).
The linear parts are the same in each of the three systems they differ only in the translation vectors. However, $\dim_{\rm H} \Lambda^l=\dim_{\rm B} \Lambda^l=\dim_{\rm Aff} \Lambda^l$,
$\dim_{\rm H} \Lambda^m<\dim_{\rm B} \Lambda^m=\dim_{\rm Aff} \Lambda^m$ and
$\dim_{\rm H} \Lambda^r<\dim_{\rm B} \Lambda^r<\dim_{\rm Aff} \Lambda^r$, where the affinity dimension $\dim_{\rm Aff} $ plays the same rolle here as the similarity dimension in the case of self-similar sets and it will be defined in Section \ref{y79}.

Moreover, if $d^l$, $d^m$ and $d^r$ are the Hausdorff dimension
of $\Lambda^l$,$\Lambda^m$ and $\Lambda^r$ respectively, then
$$
0<\mathcal{H}^{d^l}(\Lambda^l)<\infty ,\quad
\mathcal{H}^{d^m}(\Lambda^m)=\mathcal{H}^{d^r}(\Lambda^r)=\infty .
$$
\end{example}

For simplicity here we explain everything on the plane but the definitions and discussions in $\mathbb{R}^d$ for $d \geq 3$ are similar. (See e.g. \cite[Section 9.4]{falconer2004fractal} for  the introduction in higher dimension.)

We can define the self-affine measures exactly as we defined self-similar measures in Section \ref{z41}. That is for a probability vector $\mathbf{p}=(p_1, \dots ,p_m)$ the self-affine measure corresponding to $\mathcal{F}$ and $\mathbf{p}$ is
 \begin{equation}\label{z40}
   \nu=\nu_{\mathcal{F},\mathbf{p}}:=\Pi_*(\mathbf{p}^{\mathbb{N}}).
 \end{equation}

\end{definition}

\subsection{Singular value function, affinity dimension,
 Falconer's Theorem}\label{y79}

Most of the basic concepts of this field were introduced by Falconer \cite{falconer1988hausdorff}.
The \emph{singular value function} $\phi^s(A)$ of a matrix $A$
 is defined by
\begin{equation}\label{z28}
  \phi^s(A)=\begin{cases}
\alpha_{\lceil s\rceil}(A)^{s-\lfloor s\rfloor}\prod_{j=1}^{\lfloor s\rfloor}\alpha_j(A) & \text{if }0\leq s\leq\mathrm{rank}(A), \\
|\det(A)|^{s/\mathrm{rank}(A)} & \text{if }\mathrm{rank}(A)<s,
\end{cases}
\end{equation}
where $\alpha_i(A)$ denotes the $i$th singular value of $A$. On the plane, for a non-singular matrix $A$ this is simply
\begin{equation}\label{z29}
  \phi^s(A):=
  \left\{
    \begin{array}{ll}
    \alpha_1(A)  , & \hbox{if $s\leq 1$;} \\
    \alpha_1(A)\alpha_2^{s-1}(A)  , & \hbox{if  $1 \leq s \leq 2$;} \\
   ( \alpha_1(A)\alpha_2(A))^{s/2}  , & \hbox{if $s \geq 2$.}
    \end{array}
  \right.
\end{equation}
Using the singular value function Falconer \cite{falconer1988hausdorff} defined the affinity dimension $\dim_{\rm Aff} \Lambda$ as the root of the subadditive pressure formula
\begin{equation}\label{z04}
  P_{A_1, \dots A_m}(\dim_{\rm Aff} \Lambda)=0,
\end{equation}
where the function $s\mapsto P_{A_1, \dots A_m}(s )$ is defined in the Appendix Example \ref{z27}.
This is the  value of the Hausdorff dimension of $\Lambda$ in most of the cases.
\begin{theorem}[Falconer]\label{z10}
Fix the $d\times d$ non-singular matrices $A_1, \dots ,A_m$ in any particular ways satisfying
$\max\limits_{1 \leq i \leq m}\|A_i\|<1/2$.
For every $\mathbf{t}=(t_1, \dots ,t_m)\in\mathbb{R}^{md}$ we consider
the following self-affine IFS on $\mathbb{R}^d$:
 $\mathcal{F}^{\mathbf{t}}:=\left\{f_i(x):=A_ix+t_i\right\}_{i=1}^{m}$,
where the translations $\mathbf{t}=(t_1, \dots ,t_m)$ are considered as parameters. Then
  $\dim_{\rm H} \Lambda=\dim_{\rm B} \Lambda=\dim_{\rm Aff}\Lambda $ for Lebesgue almost all choices of  $(t_1, \dots, t_m)\in\mathbb{R}^{dm}$.
\end{theorem}

        \bigskip
\section{Ergodic measures for a self-affine IFS}
     Let $\mathcal{F}$ be a self-affine IFS  as in Definition \ref{z39}.
        Then for an arbitrary ergodic measure $\nu$ on $\Sigma$ we have
       \begin{equation}\label{z52}
    \chi_k(\nu):=   \chi_k(\Pi_*\nu):=   \lim\limits_{n\to\infty} \frac{1}{n}
          \log\alpha_k(A_{i_1}\cdots A_{i_n}).
        \end{equation}
         where $\alpha_k(B)$ is the $k$-th singular value of the matrix $B$.

  In high generality we know only almost all type formulas
          for the Hausdorff dimension of $\Pi_*\nu$.
          Namely, we consider the translations $\mathbf{t}=(t_1, \dots ,t_m)$
          as parameters (as in Theorem \ref{z10}) in the self affine IFS of the form \eqref{z38} and we write $\mathcal{F}^{\mathbf{t}}$ instead of $\mathcal{F}$,
           $\Pi^{\mathbf{t}}$ instead of $\Pi$ and $\Pi^{\mathbf{t}}_*\nu$ instead of $\Pi_*\nu$. Then \cite[Theorem 1.9]{JordanPollicottSimon07}
gives an  analogous assertion to  Falconer's theorem (Theorem \ref{z10}) for self-affine measures instead of self-affine sets:
\begin{theorem}[Jordan Pollicott and Simon] Let $\nu$ be an arbitrary ergodic measure on $\Sigma=\left\{1, \dots ,m\right\}^{\mathbb{N}}$.
  If  $\max\limits_{1 \leq i \leq m}\|A_i\|<1/2$
        then for almost all $\mathbf{t}$ (w.r.t. the $m \cdot d$-dimensional Lebesgue measure)
        we have
       \begin{equation}\label{z43}
         \dim_{\rm H} (\Pi^{\mathbf{t}}_*\nu)
         =\min\left\{d,D(\nu)\right\},
       \end{equation}
where $D(\nu)$ is the Lyapunov dimension for the ergodic measure
$\nu$ defined below.
\end{theorem}
\begin{definition}
Let $\mathcal{F}$ be a self-affine IFS  as in Definition \ref{z39}.
        Then for an arbitrary ergodic measure $\nu$ on $\Sigma$
  \begin{equation}\label{z20}
    D(\nu):=
         k+\frac{h(\nu)+\chi_1(\nu)+\cdots+\chi_k(\nu)}{-\chi_{k+1}(\nu)},
  \end{equation}
if $k=k(\nu)=\max\left\{i:0<h(\nu)
+\chi_1(\nu)+\cdots+\chi_i(\nu)
\right\} \leq d$.
On the other hand, if        $0<h(\nu)
+\chi_1(\nu)+\cdots+\chi_d(\nu)$ then we define
\begin{equation}\label{z14}
   D(\nu):=d \cdot
\frac{h(\nu)}{-(\chi_1(\nu)+\cdots+\chi_d(\nu))}.
\end{equation}
We call $D(\nu)$ the Lyapunov dimension of the measure $\nu$.
\end{definition}

\begin{example}\label{z11}
 In this paper we mostly work on the plane ($d=2$). In this case
\begin{equation}\label{z12}
D(\nu)=
\left\{
  \begin{array}{ll}
  \frac{h(\nu)}{|\chi_1(\nu)|}  , & \hbox{if
$h(\nu) \leq |\chi_1(\nu)| $ ;} \\
1+\frac{h(\nu)-|\chi_1(\nu)|}{|\chi_2(\nu)|}    ,
 & \hbox{if $ |\chi_1(\nu)| \leq h(\nu) \leq |\chi_1(\nu)|+|\chi_2(\nu)|$;} \\
 2 \cdot \frac{h(\nu)}{|\chi_1(\nu)|+|\chi_2(\nu)}|   , & \hbox{if $|\chi_1(\nu)|+|\chi_2(\nu)| \leq h(\nu)$.}
  \end{array}
\right.
 \end{equation}
\end{example}

Recently there have been a number of very significant achievements on this field.
Here we mention only one of them. B\'ar\'any, Hocfhman and Rapaport   \cite[Theorem 1.2]{BHR} computed the Hausdorff dimension of self-affine measures under some mild conditions. They obtained this by
combining the entropy growth theorem by Hochman \cite{hochman2012self} with the method of B\'ar\'any and K\"aenm\"aki \cite{barany2017ledrappier} about the dimension of the projections of self-affine measures, that they got by an application of the Furstenberg measures.
\subsection{Self-affine measures}
\begin{definition}\label{z08}
 Let  $\mathcal{F}:=\left\{f_i(x):=A_ix+t_i\right\}_{i=1}^{m}$ be a self-affine IFS on $\mathbb{R}^d$ and let $\mathbf{p}$ be a probability vector. Then the corresponding self-affine measure can be defined exactly as we defined the self-similar measures. That is
\begin{equation}\label{z06}
  \nu=\nu_{\mathcal{F},\mathbf{p}}:=\Pi_*\left(\mathbf{p}^{\mathbb{N}}\right),
\end{equation}

\end{definition}
In their very recent seminal paper B\'ar\'any, Hochman and Rapaport \cite[Theorems 1.1 and 1.2]{BHR} proved the following


\begin{theorem}[B\'ar\'any, Hochman and Rapaport]\label{z01}
   Let
 $\mathcal{F}:=\left\{f_i(x):=A_ix+t_i\right\}_{i=1}^{m}$ be a self-affine IFS on $\mathbb{R}^2$ which satisfies both of the following conditions:
 \begin{description}
   \item[(a)] the strong open set condition (see Definition \ref{z42}) and
   \item[(b)] The normalized linear parts $\left\{A_i/\sqrt{|\det A_i|}\right\}_{i=1}^{m}$
 generate a non-compact and totally irreducible subgroup of $GL_2(\mathbb{R}^d)$ (that is they  do not preserve any finite union of non-trivial linear
spaces,)
\end{description}
Then for an arbitrary probability vector $\mathbf{p}$ we have
\begin{equation}\label{z02}
  \dim_{\rm H} \nu_{\mathcal{F},\mathbf{p}}=D(\nu_{\mathcal{F},\mathbf{p}}) \mbox{ and }
\dim_{\rm H} \Lambda=\dim_{\rm B}\Lambda=\dim_{\rm Aff}\Lambda,
\end{equation}
where $\Lambda$ is the attractor of $\mathcal{F}$
   and we remind the reader that the affinity dimension $\dim_{\rm Aff}$ was defined in \eqref{z04}.
\end{theorem}
This theorem does not cover the case of those self affine IFS  for which all of the mappings have lower triangular linear parts. However,
 the same authors proved in \cite[Proposition 6.6]{BHR}

\begin{theorem}[B\'ar\'any, Hochman and Rapaport]\label{z00}
 Let
 $\mathcal{F}:=\left\{f_i(x):=A_ix+t_i\right\}_{i=1}^{m}$ be a self-affine IFS on $\mathbb{R}^2$ which satisfies both of the following conditions:
\begin{description}
\item[(c)] The linear parts of all of the mapping of $\mathcal{F}$ are lower triangular: \newline $A_i=\left(
                       \begin{array}{cc}
                         a_i & 0 \\
                        b_i & c_i \\
                       \end{array}
                     \right)
$ for $i=1, \dots ,m$ and
\item[(d)] $a_i<c_i$ for all $i=1, \dots ,m$.
 \end{description}
Then
 for an arbitrary probability vector $\mathbf{p}$ we have
\begin{equation}\label{z02}
  \dim_{\rm H} \nu_{\mathcal{F},\mathbf{p}}=D(\nu_{\mathcal{F},\mathbf{p}}) \mbox{ and }
\dim_{\rm H} \Lambda=\dim_{\rm B}\Lambda=\dim_{\rm Aff}\Lambda,
\end{equation}
where $\Lambda$ is the attractor of $\mathcal{F}$.

\end{theorem}

        \bigskip
\section{ Ergodic measures for Barnsley's skew product maps}
 We use the notation of Section \ref{z49}.
 Let $\mu$ be an ergodic measure for the  Barnsley's skew product map $F$, which was  defined in Section \ref{z49}. The two Lyapunov exponents $\chi_1(\mu)$ and $\chi_2(\mu)$ of $F$ are
 \begin{eqnarray*}
&&\chi_x(\mu)=\int\log\|D_{\mathrm{proj}(\mathbf{x})}
f\|
\mathrm{d}\mu(\mathbf{x})=
\sum_{i=1}^m\mu(I_i\times\mathbb{R})\log\gamma_i\text{ and }\\
&&\chi_y(\mu)=\int\log\|\partial_2g(\mathbf{x})\|
\mathrm{d}\mu(\mathbf{x})=
\sum_{i=1}^m\mu(I_i\times\mathbb{R})\log\lambda_i,
\end{eqnarray*}
  where $\mathrm{proj}(\mathbf{x})$ is the orthogonal projection of an $\mathbf{x}\in D$ to the $x$-axis and $\partial_2$ means the derivative with respect to the second coordinate.
  \begin{remark}\label{z37}
    If $0<\chi_x(\mu)\leq\chi_y
    (\mu)$ then
$$
\dim\mu=\frac{h(\mu)}{\chi_x(\mu)},
$$

Namely, the upper bound is trivial and the lower bound follows from the fact that $\mathrm{proj}_*\mu$ is $f$-invariant and ergodic and the result of Hofbauer and Raith \cite[Theorem~1]{hofbauer_raith} (see \eqref{z44}).
That is why we can restrict ourselves to the case when
\begin{equation}\label{z36}
\chi_1(\mu):= \chi_x(\mu)= \sum_{i=1}^m\mu(I_i\times\mathbb{R})\log\gamma_i
>
\chi_2(\mu):=\chi_y(\mu)=
\sum_{i=1}^m\mu(I_i\times\mathbb{R})\log\lambda_i
>0.
\end{equation}
In this case the best guess for the dimension of the $\mu$ is the so-called Lyapunov dimension to be defined below.
  \end{remark}
\begin{definition}\label{z35}
  Let  $\mu\in\mathcal{E}_F(\Lambda)$ satisfying $\chi_x(\mu)>\chi_y(\mu)>0$. We define the Lyapunov dimension
\begin{equation}\label{y97}
D(\nu):=
\left\{
  \begin{array}{ll}
  \frac{h(\nu}{\chi_y(\nu)}  , & \hbox{if
$h(\nu) \leq \chi_y(\nu) $ ;} \\
1+\frac{h(\nu)-\chi_y(\nu)}{\chi_x(\nu)}    ,
 & \hbox{if $\chi_y(\nu) \leq h(\nu) \leq \chi_x(\nu)+\chi_y(\nu)$;} \\
 2 \cdot \frac{h(\nu)}{\chi_x(\nu)+\chi_y(\nu)}   , & \hbox{if $\chi_x(\nu)+\chi_y(\nu) \leq h(\nu)$.}
  \end{array}
\right.
 \end{equation}
\end{definition}

\section{Hofbauer's Pressure}

In the previous sections (and in the appendix) we presented the dimension theory for the self-affine iterated function systems. However, the principal distinction of the Barnsley's  maps from the iterated function systems lies in the fact that the symbolic space for the Barnsley's skew product map is not a full shift. In this section we will present the most general version of thermodynamical formalism theory, developed in a series of papers by Franz Hofbauer with his co-authors. This theory is not completely general, it assumes the system comes form piecewise monotone maps of the interval, but this assumption is satisfied in our situation.

Let us remind the notations. Our base map $f:[0,1]\to [0,1]$ is piecewise monotone: we can divide the interval $[0,1]$ into finitely many closed intervals with disjoint interiors $[0,1]=\bigcup_1^m I_i$. We denote by $\mathfrak{S}$ the set of endpoints of intervals $I_i$. We assume that $f|_{I_i^o}$ is continuous and monotone (strictly increasing or strictly decreasing) on $I_i^o$. We define $f_i$ as the extension of $f|_{I_i^o}$ by continuity to the endpoints of $I_i$.

In order that the symbolic expansion of the system (to be defined below) is compact, we need to take a formal modification of the maps. We would like to consider $f_i$ as the restriction of $f$ to $I_i$. Naturally, such a definition can in general lead to the map being doubly defined on some points in $\mathfrak{S}_\infty$, but this set is countable. Formally speaking, if for a point $x\in \mathfrak{S}$ the left and right limits of $f$ disagree then we define $f(x_-)=\lim_{z\nearrow x} f(z)$ and $f(x_+)=\lim_{z\searrow x} f(z)$. We then proceed to inductively double all the preimages of $x$. For a point $y\in f^{-1}(x), y\notin \mathfrak{S}$ we define: if $f$ is increasing at $y$ then $f(y-)=x_-$ and $f(y_+)=x_+$, otherwise $f(y-)=x_+$ and $f(y_+)=x_-$. And for a point $y\in f^{-1}(x), y\in \mathfrak{S}$: if $\lim_{z\nearrow y} f(z) =x$ and $f$ is increasing in $(y-\varepsilon, y)$ then $f(y_-)=x_-$, if it is decreasing then $f(y_-)=x_+$, if $\lim_{z\searrow y} f(z) =x$ and $f$ is increasing in $(y, y+\varepsilon)$ then $f(y_+)=x_+$, if it is decreasing then $f(y_+)=x_-$. We set the natural topology: at each doubled point $x$ $\lim_{z\nearrow x} z = x_-, \lim_{z\searrow x} z = x_+$. We also redefine the partition intervals: if $I_i=[x,y]$ and one or both of the endpoints are doubled then we set $I_i=[x_+, y_-]$.

Observe that the resulting set is not an interval anymore, but a Cantor set - but with a natural projection onto the interval, which is 2-1 on a countable set and 1-1 elsewhere. The well-known special case of this construction: consider the interval $[0,1]$ with the map $f(x)=2x (\mod 1)$ and divide each dyadic point into two. That is, $1/2 = 0.10000..._2 = 0.01111..._2$, we formally define $(1/2)_-=0.01111..._2$ and $(1/2)_+=0.10000..._2$ -- and the same for all the other dyadic points. The result is a full shift on two symbols, which is conjugate (modulo a countable set) to the original map.

Note that for the piecewise monotone map the minimal possible partition is given by the intervals of monotonicity of $f$, but we can freely subdivide the intervals $I_i$ further, and the resulting maps will also belong to considered class. In particular, we can freely demand that for any given continuous potential $\varphi:[0,1]\to \mathbb{R}$ its variation $\sup \varphi - \inf \varphi$ is arbitrarily small on each $I_i$.

Let $A$ be a compact, $f$-invariant, $f$-transitive set. For the rest of the section, our dynamical system will be the restriction of $f$ to $A$.

Let $\widetilde{\Sigma}\subset \{1,\ldots,m\}^{\mathbb{N}}$ be the symbolic system of our dynamics, defined as the set of sequences $\omega\in \{1,\ldots,m\}^{\mathbb{N}}$ such that there exists $x\in A$ such that for $n=0,1,\ldots$

\[
f^n(x) \in I_{\omega_n}.
\]
One can check that $\widetilde{\Sigma}$ is a {\it subshift}, that is a $\sigma$-invariant and closed subset of $\{1,\ldots,f\}^{\mathbb{N}}$. The sequence $\omega$ will be called {\it symbolic expansion} of $x$, $x$ will be called {\it representation} of $\omega$. We will write  $x=\pi(\omega)$. We will assume the partition $\{I_i\}$ is {\it generating}, that is each $\omega\in\widetilde{\Sigma}$ has unique representation. This always holds if $f$ is expanding.

For any finite word $\tau^n\in \{1,\ldots,m\}^n$ denote by $C[\tau^n]$ the set of points $x\in A$ such that $\pi^{-1}(x)$ begins with $\tau^n$. This set will be called $n$-th level {\it cylinder}. The set of $n$-th level cylinders will be denoted $D_n$. For $x\in A$, let $C_n(x)$ be the $n$-th level cylinder containing $x$. Denote $d_n(x) = \mathrm{diam} C_n(x)$ and $\varphi_n(x) = \sup \{\varphi(y)-\varphi(z); y,z\in C_n(x)\}$. We have

\[
\lim_{n\to\infty} d_n(x) = \lim_{n\to\infty} \varphi_n(x) =0.
\]

\begin{definition}\label{y86}
 We say that $A$ is {\it Markov} if there exists such partition $\{I_i\}$ and such $n$ that for every $n$-th level cylinder $C[\tau^n]$ its image $T(C[\tau^n])$ is a union of $n$-th level cylinders. Equivalently, $A$ is Markov if for some partition $\{I_i\}$ the subshift $\widetilde{\Sigma}$ is a {\it subshift of finite type}, that is a subshift defined as all the infinite words $\omega\in \{1,\ldots,m\}^\mathbb{N}$ that do not contain any word from some finite list of finite words.
\end{definition}

\subsection{Pressure and Markov sets}

Let $\varphi: [0,1]\to \mathbb{R}$ be a piecewise continuous potential, with the set of discontinuities contained in $\mathfrak{S}$. For the Markov systems we can define the pressure in the usual way:

\begin{equation} \label{eqn:presmarkov}
P(A, \varphi) = \lim_{n\to\infty} \frac 1n \log \sum_{C[\omega^n]\in D_n} \exp(\sup_{x\in C[\omega^n]} S_n\varphi(x)),
\end{equation}
compare \eqref{57}.
For the non-Markov systems the right hand side of this equation is still well-defined, but is considered too large for applications in dimension theory. Let us give a short explanation.

In the year 1973 Rufus Bowen \cite{bowen1973topological} gave the following definition of topological entropy: given a continuous map $f:X\to X$, where $X$ is any $f$-invariant set (not necessarily compact), let $X_n$ be the $n$-th level cylinders, then

\[
h_{\rm top}(f,X) = \inf\{s; \inf_{X\subset \bigcup E_i} \sum e^{-sn(E_i)} =0\},
\]
where the sum is taken over covers of $X$ with cylinders and for a cylinder $E$ $n(E)$ denotes its level. Geometrically, the Bowen's definition of topological entropy is similar  to the Hausdorff dimension as the usual definition \eqref{z77} is similar to the box counting dimension -- or more precisely, the Bowen's definition is the Hausdorff dimension and \eqref{z77} is the box counting dimension, both calculated in a special metric (so-called dynamical metric). Still, Bowen proved that for compact $X$ the two definitions are equal, while for noncompact the Bowen's definition gives in general a smaller number. For example, for a countable set $X$ the Bowen's entropy is always 0.

Our set $A$ is compact, so there is no disagreement about what $h_{\rm top}(f,A)$ is. However, even though the pressure is heuristically a very similar object to the topological entropy (in both cases we are just counting how many trajectories the system has, except in the case of pressure we count the trajectories with some weights, given by the potential), there is no analogue of Bowen's theorem. Thus, we can always define the pressure by formula \eqref{eqn:presmarkov}, but it is only an upper bound for the correct formula -- which we do not know.

Except for the Markov systems. For a Markov system each $n$-th level cylinder is {\it large}, in the sense that there exists $\delta>0$ such that for every $C[\omega^n]\in D_n$ we have

\[
\mathrm{diam} f^n(C[\omega^n]) > \delta.
\]

It is not necessarily so for non-Markov systems: some $n$-th level cylinders might be very tiny (they will be not only $n$-th level cylinders but also $n+1,\ldots,n+\ell$-th level cylinders, for some possibly large $\ell$). As the result, the sum on the right hand side of \eqref{eqn:presmarkov} overstates their importance (counting them as $n$-th level cylinders while they would be counted as $n+\ell$-th level cylinders by Bowen). Thus, Franz Hofbauer in \cite{hofbauer2010multifractal} gave a better definition of pressure:

\begin{equation} \label{eqn:presnonmarkov}
P(A, \varphi) = \sup_{B\subset A, B {\rm Markov}} P(B, \varphi),
\end{equation}
where $P(B, \varphi)$ is given by \eqref{eqn:presmarkov}. For Markov $A$ \eqref{eqn:presnonmarkov} gives the same value as \eqref{eqn:presmarkov}.
We note that it is still an open question whether the formula \eqref{eqn:presnonmarkov} can be strictly smaller than \eqref{eqn:presmarkov} for non-Markov $A$.

\subsection{Conformal measure and small cylinders}

We finish the section with two more important results of Franz Hofbauer. The first of them was obtained together with Mariusz Urba\'nski \cite{hofbauer1994fractal}. We will call a probabilistic measure $\mu$ defined on $A$ {\it conformal} for the potential $\varphi$ if for every $n$ and for every $C[\omega^n]\in D_n$ we have

\[
\mu(T C[\omega^n]) = \int_{C[\omega^n]} e^{P(A, \varphi)-\varphi} d\mu.
\]
As the partition is generating, this formula can be iterated:

\[
\mu(T^n C[\omega^n]) = \int_{C[\omega^n]} e^{nP(A, \varphi)-S_n\varphi} d\mu.
\]

\begin{theorem}[Hofbauer, Urba\'nski]
Let $A$ be topologically transitive, compact, $T$-invariant set of positive entropy. Then for every piecewise continuous potential $\varphi$ there exists a nonatomic conformal measure $\mu(A, \varphi)$ with support $A$.
\end{theorem}

The second result of Hofbauer, from \cite{hofbauer2010multifractal}, provides a way of estimating the set of points $x\in A$ such that for every $n$ the cylinder $C_n(x)$ is not large. Denote

\[
N_\rho(A, \mu) = \{x\in A; \limsup_{n\to\infty} \mu(T^n C_n(x))\leq \rho\}.
\]

Denote also by $D(\alpha)$ the set of points $x\in A$ with Lyapunov exponent $\alpha$. We remind that $\varphi_1(x)$ denotes the variation of potential $\varphi$ in first level cylinder containing $x$.

\begin{lemma}[Hofbauer]
For every $\alpha> \sup_x (\log |F'|)_1(x)$,
\[
\lim_{\rho\to 0} \dim_H (N_\rho \cap D(\alpha)) =0.
\]
\end{lemma}

We note that $\sup_x (\log |F'|)_1(x)$ can be arbitrarily decreased by considering subpartitions of $\{I_i\}$.

\section{The dimension of Barnsley's repellers}
First we recall the basic definitions.
\subsection{The basic definitions}
First we recall the definition of Barnsley's skew product maps: Given
$\left\{I_i\right\}_{i=1}^{m}$ which is a partition of $[0,1]$. Let $D_i:=I_i\times \mathbb{R}$. For $(x,y)\in D_i$ we defined
$F_i(x,y):=(f_i(x),g_i(x,y))$, where $f_i:I_i\to J_i \subset [0,1]$ onto, and
\begin{equation}\label{y95}
 f_i(x):=\gamma_ix+v_i,\  g_i(x,y)=a_ix+\lambda_iy+t_i,\
 |\lambda_i|,|\gamma_i|>1, \quad t_i,v_i\in\mathbb{R}.
 \end{equation}
Also recall that we define $f(x):=f_i(x)$ if $x\in I_i$.
The set of admissible words is defined as
	\begin{equation}\label{eq:symbshift}
X:=	\mathrm{cl}\left\{(i_1,i_2,\dots)\in\Sigma:\exists x\in I\text{ such that }\forall n\geq0,\ f^n(x)\in I_{i_n}^o\right\},
	\end{equation}
where $\mathrm{cl}(A)$ is the closure of the set $A \subset \Sigma:=\left\{1, \dots ,m\right\}^{\mathbb{N}}$ in the usual topology on $\Sigma$.
\begin{definition}\label{y94}
  We say that $f$ is \textit{Markov} if $f(\overline{I_i})$ is equal to a finite union of elements in $\{\overline{I_i}\}_{i=1}^m$ for every $i=1,\ldots,m$.
\end{definition}

\subsection{Diagonal and essentially non-diagonal system}

Since the maps $F_i$ are affine the derivatives $DF_i$ are constant lower triangular matrices
$$
DF_i:=\left(
        \begin{array}{cc}
          \gamma_i & 0 \\
          a_i & \lambda_i \\
        \end{array}
      \right).
$$
However, it is very important if the derivative matrices are diagonal or essentially non diagonal along the dynamics since the proofs that work for the essentially non-diagonal case do not work for the diagonal ones and we need different assumptions in these different cases.

\begin{definition}\label{y92}
We say that
\begin{description}
  \item[(a)] $F$ is diagonal
  if all the matrices $DF_i$ are diagonal.
  \item[(b)] $F$ is essentially diagonal if the system of matrices
$\left\{DF_i\right\}_{i=1}^{m}$, simultaneously diagonizable.
 This holds if
\begin{equation}\label{y93}
\frac{\gamma_i-\lambda_i,}{a_i}=
\frac{\gamma_j-\lambda_j}{a_j}, \quad \forall i,j\in\left\{1, \dots ,m\right\}.
\end{equation}
  \item[(c)] $F$ is essentially non-diagonal along the dynamics
if there are admissible  words $\pmb{\omega},\pmb{\tau},\in X$ and another word $\pmb{\eta}$ such that $\pmb{\omega}\pmb{\eta}\pmb{\tau}\in X$ such that
\begin{enumerate}
  \item both  $f_{\pmb{\omega}}$ and $f_{\pmb{\tau}}$ have fixed points
  \item $\left\{DF_{\pmb{\omega}},DF_{\pmb{\tau}}\right\}$ are not simultaneously diagonizable. That is for
$$
DF_{\pmb{\omega}}=\left(
                    \begin{array}{cc}
                      \gamma_{\pmb{\omega}} & 0 \\
                      a_{\pmb{\omega}} & \lambda_{\pmb{\omega}} \\
                    \end{array}
                  \right)
\quad \mbox{ and } \quad
DF_{\pmb{\tau}}=\left(
                    \begin{array}{cc}
                      \gamma_{\pmb{\tau}} & 0 \\
                      a_{\pmb{\tau}} & \lambda_{\pmb{\tau}} \\
                    \end{array}
                  \right)
$$
we have
$$
\frac{\gamma_{\pmb{\omega}}-\lambda_{\pmb{\omega}}}{a_{\pmb{\omega}}}
\ne
\frac{\gamma_{\pmb{\tau}}-\lambda_{\pmb{\tau}}}{a_{\pmb{\tau}}}.
$$
\end{enumerate}

\end{description}
The reason for this restrictive definition in (c) is that during the proof we approximate by Markov sub-systems and we need to guarantee that even the approximating Markov sub-system remains essentially non-diagonal.
\end{definition}

\subsection{Markov pressure and Hofbauer Pressure}\label{y88}
Using the notation of \eqref{z48},
we introduce potential:
\begin{equation}\label{eq:potential}
\varphi^s(x)=\begin{cases}
-s\log| \lambda_i|& \text{if }0\leq s\leq 1, \\
-\left(\log|\lambda_i|+(s-1)\log|\gamma_i|\right) & \text{if }1<s\leq 2.
\end{cases}
\end{equation}
\begin{definition}[$P(s,B)$]\label{y83}
Let $s>0$ and $B \subset [0,1]$ be a Markov subset.
Recall that in
  \eqref{eqn:presmarkov} we defined the pressure $P(B,\varphi)$ for Markov subset $B \subset [0,1]$ and potential $\varphi$. Using this definition we can define
\begin{equation}\label{y84}
  P(s,B):=P(B,\varphi^s).
\end{equation}
\end{definition}

The following lemma helps to get better understanding:
\begin{lemma}\label{y85}
  Assume that $B \subset [0,1]$ is Markov of type-1 set. That is for every $i,j\in\left\{1, \dots ,m\right\}$
either $I_j\cap B \subset f(I_i\cap B)$ or $(I_j\cap B)\cap f(I_i\cap B) =\emptyset $. Then
	$$
	A_{i,j}^{(s)}=\begin{cases}
 (1/\lambda_i) \cdot (1/\gamma_i)^{s-1}
 & \text{if }I_j\cap B\subseteq f(I_i\cap B)\\
	0 & \text{otherwise.}\end{cases}
	$$
Then $P(s,B)=\log\rho(A^{(s)})$, where $\rho(A)$ denotes the spectral radius of $A$.
\end{lemma}
We remark that every subshifts of type-$n$ can be corresponded to a type-1 subshift by defining a new alphabet, and subdividing the monotonicity intervals into smaller intervals.
\begin{definition}[$P_{\mathrm{Mar}}(s),P_{\mathrm{Hof}}(s)$]
 Now we define the functions $s\mapsto P_{\mathrm{Mar}}(s)$ and
$s\mapsto P_{\mathrm{Hof}}(s)$ as follows:
\begin{description}
  \item[(a)] If $f$ is Markov then we write
$P_{\mathrm{Mar}}(s):=P(s,[0,1])$
  \item[(b)] If $f$ is none Markov then we write
\begin{equation}\label{y82}
 P_{\mathrm{Hof}}(s):=
\sup\limits_{B \subset [0,1],\ B\  \mathrm{ Markov}}
  P(s,B).
\end{equation}
\end{description}

\end{definition}

\subsection{The main results}\label{y87}
\begin{theorem}\label{y89}
	 Suppose that
\begin{description}
  \item[(a)] $F$ is essentially diagonal,
  \item[(b)] $\gamma_i>\lambda_i$ for every $i=1,\ldots, m$,
  \item[(c)] The self-similar IFS $\{g_i^{-1}(y)=\frac{y-t_i}{\lambda_i}\}_{i=1}^M$ satisfies HESC (see Condition \ref{y96})
\end{description}
 then
	$$
\dim_H\Lambda=\dim_B\Lambda=\sup_{\mu\in\mathcal{M}_{\rm{erg}}(\Lambda)}D(\mu)=s_0,
	$$
	where $s_0$ is the unique number such that
\begin{itemize}
  \item $P_{\mathrm{Mar}}(s_0)=0$ if $f$ is Markov, otherwise
  \item  $P_{\mathrm{Hof}}(s_0)=0$.
\end{itemize}

.
\end{theorem}

\begin{theorem}\label{y91}
	Assume that $F$ is essentially non-diagonal and $f$ is a topologically transitive. If $\gamma_i>\lambda_i$ for every $i=1,\ldots, m$ then
	$$
	\dim_H\Lambda=\dim_B\Lambda=\sup_{\mu\in\mathcal{M}_{\rm{erg}}(\Lambda)}D(\mu)=s_0,
	$$
where $s_0$ is the unique number such that
\begin{itemize}
  \item $P_{\mathrm{Mar}}(s_0)=0$ if $f$ is Markov, otherwise
  \item  $P_{\mathrm{Hof}}(s_0)=0$.
\end{itemize}
\end{theorem}

\appendix

\section{Thermodynamical formalism}\label{z31}
First we introduce the subshift of finite type.
\subsection{Subshift of finite type}\label{z30}
Let  $\ybox{\Sigma =\left\{1,\dots ,m\right\}^\mathbb{N}}$ be endowed with the usual topology, which generated by the distance $\mathrm{dist}(\mathbf{i},\mathbf{j}):
=m^{-\left|\mathbf{i}\wedge\mathbf{j}\right|}$,
where
   \structure{ $$\left|\mathbf{i}\wedge\mathbf{j}\right|=\max\left\{n:\forall |\ell|\leq n, i_\ell=j_\ell\right\}.$$}
 For some $k<r$ we write
$\ybox{\left[\mathbf{i}\right]_{k,r}=\left\{\mathbf{j}\in \Sigma:
i_\ell=j_\ell,\ \forall \ell\in\left\{k,\dots ,r\right\} \right\}}
$ for the $(k,r)$ cylinder sets. If $k=1$ then we write simply $[\mathbf{i}]_r$.
Similarly,
$$[i_1, \dots ,i_n]:=\left\{\mathbf{j}\in \Sigma:i_k=j_k, \forall k=1, \dots ,n\right\}.$$ For an $\mathbf{i}\in\Sigma $ we write
\begin{equation}\label{z63}
  \mathbf{i}|_n:=(i_1, \dots ,i_n) \in (1, \dots ,m)^{n}=:\Sigma_n.
\end{equation}

    \begin{definition}[\yboxduma{subshift of finite type}]\label{z58}
   Given an  $m\times m$ matrix  $A$ of $0$'s and $1$'s.  Let\pause \
  \structure{ $
   \Sigma _A:=\left\{\mathbf{i}\in \Sigma :A_{i_k,i_{k+1}}=1,\ \forall k\in \mathbb{N}\right\}
 $} and let $\sigma$ be the left shift on $\Sigma_A$. That is $\sigma(i_1,i_2,i_3 \dots ):=(i_2,i_3. \dots )$ for every $(i_0,i_1,i_2, \dots )\in\Sigma_A$. Clearly, $\sigma(\Sigma_A)=\Sigma_A$ and $\sigma|_{\Sigma_A}$ is a homeomorphism on $\Sigma_A$. Sometimes we call $\sigma|_{\Sigma_A}$
 \emph{topological Markov chain}.
 \end{definition}
 We always assume that for every $k\in \left\{1,\dots ,m\right\}$ there exist some $\mathbf{i}\in \Sigma _A$
such that $i_0=k$.\pause \
From now on we call
\begin{itemize}
  \item \structure{ $(\Sigma,\sigma )$} a \structure{full shift} and
  \item \alert{$(\Sigma_A,\sigma )$} as \alert{subshift of finite type}.\pause \
\end{itemize}

Also for the rest of this Section  we assume that $A$ is an $m\times m$ \alert{primitive matrix}.
$$
\clrgreen{\Sigma _{A,n}:=\left\{\mathbf{i}=(i_1,\dots ,i_{n}):[i_1,\dots ,i_{n}]\cap \Sigma _A\ne\emptyset \right\}}.
$$

\subsection{Ergodic measures}
Given a measurable self-map $T$ of a measurable space $(X,\mathcal{B})$. That is $T:X\to X$ and $T^{-1}B\in\mathcal{B}$ for every $B\in\mathcal{B}$. We write
\begin{itemize}
  \item $\mathcal{M}(X)$ for the set of Borel probability measures on $(X,\mathcal{B})$,
  \item $\mathcal{M}_T(X)$ for the set of  invariant measures. That is \newline $$\mathcal{M}_T(X)=\left\{\mu\in\mathcal{M}(X)
      :\mu(A)=\mu(T^{-1}A),\ \forall A \in\mathcal{B}\right\},$$
  \item  $\mathcal{E}_T(X)$ for the ergodic measures. That is
  $$
  \mathcal{E}_T(X)=\left\{\mu\in\mathcal{M}_T(X)
      :A=T^{-1}A \Longrightarrow \mbox{ either }
      \mu(A)=0, \mbox{ or } \mu(A)=1
      \right\}.
  $$
\end{itemize}
We frequently use Birkhoff's Ergodic Theorem.
\begin{theorem}[Birkhoff's Ergodic Theorem]
Let $\mu\in\mathcal{E}_T(X)$ and  let $f\in L^1(X,\mu)$.
Then for $\mu$-almost all $x\in X$ the ergodic averages converge both in $L^1$ and pointwise:
\begin{equation}\label{z64}
  \lim\limits_{n\to\infty} \frac{1}{n}
  \sum\limits_{k=0}^{n-1}
  f(T^k(x))=\int f(x)d\mu(x).
\end{equation}
\end{theorem}

\subsection{Entropy}
One of the basic concepts of the thermodynamical formalism is the entropy. There is measure theoretical and topological entropy. Here we just present the definitions and a basic property. For further reading we recommend
 \cite{Bowen75}, \cite{Walters82} and a very detailed introduction is given in \cite{przytycki2010conformal}.

 \subsubsection{Measure theoretical entropy on $(\Sigma_A,\sigma)$ for an ergodic measure}

First we define the measure theoretical entropy on $\Sigma_A$ for an ergodic (with respect to the left shift $\sigma$) measure. (We always assume that $A$ is a primitive matrix.)

\begin{definition}[Entropy (measure theoretical)]\label{z66}
  Let $\mu$ be an ergodic measure on $\Sigma_{A}$. We can define the entropy of $\mu$ as
  \begin{equation}\label{z65}
  h(\mu):=  \lim\limits_{n\to\infty} \frac{1}{n}\sum
    \limits_{\pmb{\omega}\in{\Sigma_{A,n}}}
    \mu([\pmb{\omega}])\log\mu([\pmb{\omega}]).
  \end{equation}
\end{definition}
\begin{theorem}[Shannon Breiman McMillian Theorem]
  Let $\mu\in\mathcal{E}_\sigma(\Sigma)$. Then for $\mu$-almost all
  $\mathbf{i}\in\Sigma_A$ we have
  \begin{equation}\label{z62}
    \lim\limits_{n\to\infty}
    \frac{1}{n}\log\mu[\mathbf{i}|_n]=h(\mu).
  \end{equation}
\end{theorem}
For the proof see \cite{Bowen75}.
\begin{example}
 \begin{description}
   \item[(a)]\emph{Bernoulli shift}. Given a probability vector $\mathbf{p}:=(p_1, \dots ,p_m)$, where $p_i$ and $\sum\limits_{i=1}^{m}p_i=1$.Then we say the $\mu:=\mathbf{p}^{\mathbb{N}}$ is the \emph{Bernoulli measure} corresponding to $\mathbf{p}$. It is easy to see that
       \begin{equation}\label{z61}
         h(\mu)=-\sum\limits_{i=1}^{m}p_i\log p_i.
       \end{equation}
   \item[(b)] \emph{Markov Shift}
   Given a stochastic matrix $P=(p_{i,j})_{1 \leq i,j \leq m}$. That is $\sum\limits_{j=1}^{m}p_{i,j}=1$, $p_{i,j} \geq 0$. We assume that $P$ is primitive (it was enough to assume less).
    Then by Perron Frobenius Theorem there exists a left eigenvector $\mathbf{p}=(p_1, \dots ,p_m)$ which is a probability vector, such that $\mathbf{p}^T \cdot P=\mathbf{p}^T$, ($\mathbf{p}$ is considered as a column vector). We define the \emph{Markov measure} $\mu$ on $\Sigma$
    corresponding to $(\mathbf{p},P)$
    by
    $\mu([\pmb{\omega}]):=p_{\omega_1} \cdot p_{\omega_1,\omega_2}\cdots p_{\omega_{n_1},\omega_{\omega_n}}$,
    where $\pmb{\omega}\in\Sigma_n$ and $\pmb{\omega}=(\omega_1, \dots ,\omega_n)$. Then
    \begin{equation}\label{z60}
      h(\mu)=-\sum\limits_{i,j=1}^{m}
      p_ip_{i,j}\log p_{i,j}
    \end{equation}
   \item[(c)] \emph{Parry measure} Let $A=(a_{i,j})_{1 \leq i,j \leq m}^m$ be an primitive matrix (to assume irreduciblity was enough again) whose entries belong to $\left\{0,1\right\}$. Then we define the canonical Markov measure as follows: Let $\lambda$ be the largest (Perron-Frobenius) eigenvalue. Let $\mathbf{u}:=(u_1, \dots ,u_m)$ and
       $\mathbf{v}:=(v_1, \dots ,v_m)$ be the left and right (positive) eigenvectors  satisfying $\sum\limits_{i=1}^{m}u_i=1$ and $\sum\limits_{i=1}^{m}u_iv_i=1$ (see \cite[p. 16]{Walters82}). Then we define
       \begin{equation}\label{z59}
 p_i:=u_iv_i \mbox{ and }\quad
 p_{i,j}:=\frac{a_{i,j}v_j}{\lambda v_i}
       \end{equation}
       Let $\mu$ be the Markov measure corresponding to $(\mathbf{p},P)$.
       Then the unique measure on $\Sigma_A$ with maximal entropy is $\mu$ and
       $h(\mu)=\log\lambda$.
 \end{description}
\end{example}

\subsubsection{Topological entropy on compact metric spaces for
continuous mappings}

Now we give the definition of the topological entropy in a more general setup
(see e.g. \cite[p. 165-170]{de2012one} ).
\begin{definition}[Topological entropy]\label{z78}
  Given a homeomorphism $T$ of the  compact metric space $(X,d)$.
For $\varepsilon >0$ we say the orbits of length $n$
$$
\alert{x,T(x),\dots ,T^{n-1}(x)}\mbox{ and }
\structure{y,T(y),\dots ,T^{n-1}(y)}
$$
are \yboxduma{the same with $\varepsilon $-precision}
if
$$
\ybox{d(T^{i}(x),T^{i}(y))<\varepsilon, \quad \forall i=0,\dots ,n-1}.
$$
Fix an $\varepsilon >0$ and an $n\in \mathbb{N}$.
Let $\alert{s_n(x,\varepsilon )}$ be the maximal number of $n$-orbits which are different with $\varepsilon$-precision. Then we define the topological pressure
of $T$ by
\begin{equation}\label{z77}
  \ybox{\alert{h_{\mathrm{top}}(T)=
\lim\limits_{\varepsilon \to 0}\limsup\limits_{n\to\infty}\frac{1}{n}
\log s_n(\varepsilon )}}
\end{equation}
We remark that this is not the most common way to define the topological entropy.
\end{definition}
\begin{theorem}
  Let $T:X\to X$ be a contiuous map of a compact metric space. then $h_{\mathrm{top}}(T)=\sup\left\{h_T(\mu):\mu \mbox{ is an invariant measure for }T\right\}$.
\end{theorem}
We defined the measure theoretical entropy only on subshift of finite type. The definition in the general case is similar see e.g. \cite{Bowen75} and \cite{Walters82}.
Before we give some examples we need the following definition that will also be used later.
\begin{definition}\label{z56}
Let $T:I\to I$, where $I \subset \mathbb{R}$ is an interval.
 \begin{itemize}
  \item We say that $T$ is a piecewise monotone map if there is a finite partition of $I$ such that on every class of this partition the map  $T$ is monotone.
  \item Let $T$ be a piecewise monotone map. The the \yboxduma{lap number} $\alert{\ell(T)}$ is the number of maximal monotonicity intervals of $T$.
\end{itemize}
\end{definition}

\begin{example}\label{z57}
  \begin{description}
    \item[(a)] For a subshift of finite type $(\Sigma_A,\sigma)$ the topological entropy of $\sigma$ is $\log\lambda$, where $\lambda$ is the largest eigenvalue of the primitive $0,1$ matrix $A$.
    \item[(b)] Here we use the notation of Definition \ref{z57}.
    It follows from a theorem of Misiurewicz and Szlenk that for a piecewise monotone map $T$, we have
    \begin{equation}\label{z55}
      h(T)=\lim\limits_{n\to\infty} \frac{1}{n}\log\ell (T^n),
    \end{equation}
    where $T^n$ is the $n$-fold composition of $T$. In particular, $h(T) \leq \ell (T)$. Moreover, if $T$ is piecewise affine and   its the slope of $\pm s$ at every point (except the turning points) then $h(T)=\max\left\{0,\log s\right\}$.
  \end{description}
(See \cite{de2012one} for the proofs.)
\end{example}

\subsection{Lyapunov exponent}
To define the Lyapunov exponents we need Oseledec Theorem. The following version of Oseledec Theorem is from Krengel's book \cite[p. 42-47]{krengel2011ergodic} where the proof is also presented.
  Given a finite measure space $(\Omega
,\mathcal{A},\mu )$ and $\tau:\Omega\to\Omega$ measure preserving.  Further, $M$ denotes the set of $r\times r$
matrices. Put
$$
P_n(A,\omega ):=A(\tau ^{n-1}\omega )\cdots A(\tau \omega )A(\omega
).
$$

\begin{theorem}[Oseledec]
Legyen $A:\Omega \to M$ be measurable and we assume that
\begin{equation}\label{o3}
\log^+\|A(\cdot )\| \in L_1(\mu) .
\end{equation}
Then there exists an invariant $\Omega '\subset \Omega $ which has
full $\mu$-measure such that
\begin{enumerate}
 \item $$ \lim\limits_{n\to\infty}\left(P^*_n(A,\omega )\cdot P_n(A,\omega )
  \right)^{1/2n}=:\Lambda (\omega )$$
  exists and $\Lambda $ is a symmetric positive semidefinite matrix.
  \item Let $\exp(\lambda_1(\omega ))>\cdots >\exp(\lambda_s(\omega ))$
  are the  different eigenvalues of $\Lambda $ and let $E_{\nu} $ be
  the eigenspace of $\Lambda $
  which belongs to $\exp \lambda_\nu (\omega )$. Then for
  $$
H_\nu(\omega) :=E_s(\omega )\bigoplus E_{s-1}(\omega )\bigoplus\cdots\bigoplus E_{s+1-\nu} (\omega )
  $$
we have
\begin{equation}\label{o4}
\lim\limits_{n\to\infty}\frac{1}{n}\log%
\|P_n(A,\omega )\mathbf{v} \|=\lambda_{s+1-\nu} (\omega ),\qquad \forall
\mathbf{v}\in H_\nu (\omega )\setminus H_{\nu -1}(\omega ),
\end{equation}
where $H_0(\omega)\equiv  \emptyset $.
  \item $\omega \mapsto {\rm dim}E_\nu (\omega )$ and %
  $\omega \mapsto \lambda_\nu (\omega )$ are $\tau $-invariant
  maps and we call ${\rm dim}E_\nu (\omega )$ the multiplicity of $\lambda_{i}(\omega)$.
\end{enumerate}
\end{theorem}

\begin{definition}[Lyapunov exponenets]\label{z53}
Let $\mu$ be an ergodic measure.
Then
 it follows from (3) that for all $i=1, \dots ,s$ and for $\mu$-almost all $\omega\in\Omega$, $\lambda_i(\omega)$ and
 ${\rm dim}E_\nu (\omega )$
   are constants that we call $\lambda_i$ and $d_i$ respectively, for $1, \dots ,s$. We partition the index set
    \begin{equation}\label{z33}
      \left\{1, \dots ,r\right\}=\bigsqcup\limits_{k=1}^{s}\mathcal{I}_k,\quad
      \mathcal{I}_k:=\left\{d_1+\cdots+d_{k-1}+1,\cdots,d_1+\cdots+d_{k-1}
      +d_{k}\right\}
    \end{equation}
    Then we define the Lyapunov exponents $\chi_1 \geq \chi_2 \geq \cdots \geq \chi_r$ as follows:
  \begin{multline}\label{z32}
\underbrace{\chi_{1 }=\cdots=\chi_{d_1 }}_{:=\lambda_1}
 >
 \underbrace{\chi_{d_1+1 }=\cdots=\chi_{d_1+d_2 }}_{:=\lambda_2}
 >
 \underbrace{\chi_{d_1+d_2+1 }=\cdots=\chi_{d_1+d_2+d_3 }}_{:=\lambda_3}
 >
 \cdots     \\
 >
 \underbrace{\chi_{d_1+\cdots+d_{s-2}+1}=\cdots=\chi_{d_1+\cdots+d_{s-2}+d_{s-1} }}_{:=\lambda_{s-1}}
    >
   \underbrace{\chi_{d_1+\cdots+d_{s-1}+1}=\cdots=\chi_{d_1+\cdots+d_{s-1}+d_{s} }}_{:=\lambda_s}.
  \end{multline}
\end{definition}

\subsection{Topological pressure and Gibbs measure}
In this section we always assume that $A$ is a primitive $m\times m$ matrix and we consider the topological Markov chain (or subshift of finite type ) $(\sigma,\Sigma_A)$ as defined in  Definition \ref{z58}
    \begin{definition}[H\"older continuity]
    We say that a function $\phi :\Sigma_A \to \mathbb{R}$ is \yboxduma{H\"older continuous}\pause \  if there exists $b>0$ and $\alpha \in (0,1)$ such that\pause \
\structure{\begin{equation}\label{9}
  \alert{\mathrm{var}_k\phi :=\sup\left\{\left|\phi (\mathbf{i})-\phi (\mathbf{j})\right|:
  |\mathbf{i}\wedge\mathbf{j}|\geq k\right\}\leq b\alpha ^k}.
\end{equation}}
    \end{definition}

The set of H\"older continuous functions on $\Sigma _A$ is denoted by $\ybox{\mathcal{F}_A}$. For a $\clrgreen{\phi }\in \mathcal{F}_A$ and
$\alert{\pmb{\omega}=(\omega_1,\dots ,\omega _n)\in \left\{1,\dots ,m\right\}^n}$
\begin{equation}\label{z74}
S_n\phi (\alert{\pmb{\omega}}):=\sup\left\{\sum\limits_{\ell=0}^{n-1}
\clrgreen{\phi} (\sigma ^\ell\structure{\mathbf{j}}):\structure{\mathbf{j}}\in [\alert{\pmb{\omega}}]\cap \Sigma _A
\right\}.
\end{equation}


First observe that for any $\phi \in \mathcal{F}_A$ satisfying (\ref{9}):
and for any $\structure{\mathbf{j}},\clrgreen{\mathbf{j}'}\in [\alert{\pmb{\omega}}]$, where  $\alert{\pmb{\omega}=(\omega_1,\dots ,\omega _n)\in \Sigma _{A,n}}$ we have
\begin{equation}\label{11}
\left|  \sum\limits_{\ell=0}^{n-1}
\phi (\sigma ^\ell\structure{\mathbf{j}})
  -
  \sum\limits_{\ell=0}^{n-1}
\phi (\sigma ^\ell\clrgreen{\mathbf{j}'})\right|
  \leq \frac{b}{1-\alpha }
\end{equation}
holds for all $n$ and $\pmb{\omega}\in\Sigma_{A,n}$. This yields that the \yboxduma{topological pressure of the potential $\pmb{\clrgreen{\phi }}$} for the topological Markov shift $(\Sigma_A,\sigma)$ is
\begin{equation}\label{57}
 \ybox{P(\pmb{\clrgreen{\phi} })}:=\lim\limits_{n\to\infty}\frac{1}{n}
\log\left(\sum\limits_{\alert{\mathbf{i}}\in \Sigma_{A,n}}
\e{S_n\phi(\mathbf{i})}
\right)
\end{equation}

does not depend on which $\structure{\mathbf{j}}\in [\alert{\mathbf{i}}]$ is chosen. Let $\mathcal{M}_\sigma(\Sigma_A)$ denote the $\sigma$-invariant probability measures on $\Sigma_A$. The so-called Gibbs measure together with the topological pressure play central role in dimension theory:

\begin{theorem}[The Existence of Gibbs Measure Theorem]\label{19}
Suppose that
\begin{itemize}
  \item  $A$ is primitive and
  \item \structure{$\phi \in \mathcal{F}_A$}.
\end{itemize}\pause \
Then there exists a unique  $\mu \in \mathcal{M}_\sigma (\Sigma _A)$\pause \
for which  \structure{$\exists c_1,c_2>0$}\pause \  such that for  \structure{$\forall \mathbf{i}\in \Sigma_A $} and  \structure{$\forall \ell $}:\pause \
\alert{\begin{equation}\label{10}
\alert{c_1\leq
\frac{\mu \left(
\left[\mathbf{i}\right]_{\ell }
\right)}
{\exp\left(-\ell  \cdot P(\phi)+S_\ell  \phi (\mathbf{i})\right)}
\leq c_2,}
\end{equation}}
where recall that we defined $\small{\ybox{\left[\mathbf{i}\right]_{\ell }=
\left\{\mathbf{j}\in \Sigma _A:
i_k=j_k,\ \forall k\in\left\{1,\dots ,\ell \right\} \right\}}.}
$ It can be proved that $\mu$ is mixing, consequently ergodic.
\end{theorem}
We say that $\mu$ is the Gibbs measure for the potential $\phi$.
For the proof see \cite{Bowen75}.

\subsection{The root of the  pressure formula}\label{z76}
Let $\mathcal{F}$ be a conformal IFS on $\mathbb{R}$ as in definition \ref{z79} and we assume that the SSP holds. That is $f_i([0,1])\cap f_j([0,1])= \emptyset $ for all $i\ne j$. Let $\phi_s:\Sigma\to \mathbb{R}$ be
\begin{equation}\label{z67}
  \phi_s(\mathbf{i}):=
  \log |f'_{i_1}(\sigma\mathbf{i})|^s.
\end{equation}
Then for every $\mathbf{i}\in\Sigma$ and $n$ we have
\begin{equation}\label{z75}
  \phi_s(\sigma^{n-1}\mathbf{i})+\cdots+\phi_s(\sigma\mathbf{i})+\phi_s(\mathbf{i})
  =\log |f'_{i_1 \dots i_n}(\Pi(\sigma^n\mathbf{i}))|^s.
\end{equation}
Using this and the Bounded Distortion Property, we obtain that for every $n$ and for every $\pmb{\omega}\in\Sigma_n:=\left\{1, \dots ,m\right\}^n$
\begin{equation}\label{z73}
  s\log c_1<
  \left|
  S_n\phi_s (\alert{\pmb{\omega}})-\log |f'_{i_1 \dots i_n}(\Pi(\sigma^n\mathbf{i}))| ^s
  \right|
 < s\log c_2.
\end{equation}
Hence we get
\begin{equation}\label{z72}
 P(s):= P(\phi_s)=\lim\limits_{n\to\infty} \frac{1}{n} \log
  \sum\limits_{|\pmb{\omega}|=n}
   |f'_{i_1 \dots i_n}(0)|^s,
\end{equation}
It is easy to see that the function $s\mapsto P(\phi_s)$ is positive at zero, negative at $1$, continuous  and strictly decreasing. So it has a unique zero in $(0,1)$. Let us denote this unique zero by $s_{0}$. That is
\begin{equation}\label{z71}
  P(s_{0})=0.
\end{equation}
This is the reason that we say that $s_{0}$ is the root of the pressure formula.

Let $\mu$ be the Gibbs measure for the potential $\phi_{s_{0}}$. Then for every $n$, $\pmb{\omega}\in\Sigma_n$, and $x\in(0,1)$ we have
\begin{equation}\label{z70}
  c_1c_3<
  \frac{\mu([\pmb{\omega}])}{|f'_{\pmb{\omega}}(x)|^{s_0}}
 < c_2c_4.
\end{equation}

\section{Subadditive Pressure}
Falconer introduced subadditive pressure  in \cite{falconer1988hausdorff} and in a more explicit form in \cite[Section 3]{Falc_94_sub_add}.
\begin{definition}[Subadditive pressure]\label{z25}
  Assume that $\psi_n:\Sigma_A\to\mathbb{R}$ , $n=1,2, \dots $ satisfy the following three conditions:
\begin{description}
  \item[(a)] $\psi_{n+m}(\mathbf{i}) \leq \psi_n(\mathbf{i})+\psi_m(\sigma^m\mathbf{i})$, $n,m\in\mathbf{N}$
  \item[(b)]  There exists an $ a>0$ such that
$\left|\frac{1}{n}\psi_n(\mathbf{i})\right| \leq a$, for all $\mathbf{i}\in\Sigma_A, n\in\mathbf{N}$
  \item[(c)]  There exists an $ a>0$ such that
$\left|
\psi_n(\mathbf{i})-\psi_n(\mathbf{j})
\right| \leq b  $
for all $n\in\mathbb{N}$ and $\mathbf{i},\mathbf{j}\in\Sigma_A$.
\end{description}
Foe every $\pmb{\omega}\in\Sigma_{A,n}$ we fix an arbitrary $\mathbf{i}_{\pmb{\omega}}\in[\pmb{\omega}]$. Then
the subadditive pressure associated to $\left\{\psi_n\right\}$ is
\begin{equation}\label{z23}
  P(\left\{\psi_n\right\}):=
\lim\limits_{n\to\infty}
\frac{1}{n}\log\sum\limits_{\pmb{\omega}\in\Sigma_{A,n}}
\exp\left(
\psi_n(\mathbf{i}_{\pmb{\omega}})
\right)
=
\inf\limits_n
\frac{1}{n}\log\sum\limits_{\pmb{\omega}\in\Sigma_{A,n}}
\exp\left(
\psi_n(\mathbf{i}_{\pmb{\omega}})
\right).
\end{equation}
\end{definition}
The the second equality is verified in \cite[Section 3]{Falc_94_sub_add} is a slightly different setup. The connection to the additive pressure is that
\begin{equation}\label{z22}
  P(\left\{\psi_n\right\})=
\lim\limits_{N\to\infty}
\frac{1}{N}
P\left(\sigma^N,\psi_N\right)=
\inf\limits_N  \frac{1}{N}
P\left(\sigma^N,\psi_N\right),
\end{equation}
where $P\left(\sigma^N,\psi_N\right)$ is the additive pressure (defined in \eqref{57}) for the potential $\psi_N$ on the topological Markov shift
$(\Sigma_A,\sigma^N)$.

Most commonly we use this
in the following special case:
\begin{example}\label{z24}
  In the case of the additive pressure
$\psi_n(\mathbf{i})=\sum\limits_{k=0}^{n-1}f(\sigma^n\mathbf{i})$ for a continuous function $f:\Sigma_A\to\mathbb{R}$.
\end{example}

\begin{example}\label{z27}
 Given contracting non-singular $d\times d$
matrices $A_1, \dots ,A_m$ (the linear part of a self-affine IFS  of the form \ref{z38}). Then for every $s \geq 0$ we define
\begin{equation}\label{z26}
  \psi^s_n:\Sigma_A\to\mathbb{R},
\qquad
\psi^s_n(\mathbf{i}):=\log\phi^s(A_{i_1}\cdots A_{i_n}) \mbox{ and }
P(s):=P_{A_1 \dots A_n}(s):=P\left(\left\{\psi_{n}^{s}\right\}\right).
\end{equation}
where $\phi^s$ is the singular value function defined in \eqref{z29}. It is immediate that the function $ s\mapsto P_{A_1 \dots A_n}(s)$ is strictly decreasing, continuous, positive at zero and negative at any $s$ which is large enough. So, it has a unique zero $s_{A_1 \dots A_n}>0$. That is
\begin{equation}\label{z21}
  P_{A_1 \dots A_n}(s_{A_1 \dots A_n})=0.
\end{equation}
\end{example}

\bibliographystyle{plain}
\bibliography{barnsley_08_08}

\end{document}